\documentclass[12pt]{amsart}
\usepackage[utf8]{inputenc}
\usepackage{enumerate,enumitem,amsmath,amssymb,mathrsfs,stmaryrd}
\usepackage{latexsym,epsf,graphicx,comment,appendix}

\newfont{\bb}{msbm10 at 12pt}
\newfont{\tbb}{msbm10 at 8pt}

\baselineskip 
\usepackage {amsmath}
\usepackage {amsthm}
\usepackage{times}
\usepackage{amscd}
\usepackage{epsf}

\numberwithin{equation} {section}

\begin{document}
\mbox{}\vspace{0.2cm}\mbox{}

\providecommand{\keywords}[1]
{
  \small	
  \textbf{\textit{Keywords---}} #1
}

\theoremstyle{plain}\newtheorem{lemma}{Lemma}[section]
\theoremstyle{plain}\newtheorem{proposition}{Proposition}[section]
\theoremstyle{plain}\newtheorem{theorem}{Theorem}[section]

\theoremstyle{plain}\newtheorem*{theorem*}{Theorem}
\theoremstyle{plain}\newtheorem*{main theorem}{Main Theorem} 
\theoremstyle{plain}\newtheorem*{lemma*}{Lemma}
\theoremstyle{plain}\newtheorem*{claim}{Claim}

\theoremstyle{plain}\newtheorem{example}{Example}[section]
\theoremstyle{plain}\newtheorem{remark}{Remark}[section]
\theoremstyle{plain}\newtheorem{corollary}{Corollary}[section]
\theoremstyle{plain}\newtheorem*{corollary-A}{Corollary}
\theoremstyle{plain}\newtheorem{definition}{Definition}[section]
\theoremstyle{plain}\newtheorem{acknowledge}{Acknowledgment}
\theoremstyle{plain}\newtheorem{conjecture}{Conjecture}

\begin{center}
\rule{15cm}{1.5pt} \vspace{.4cm}

{\bf\Large Linear potentials and applications in conformal geometry} 
\vskip .3cm

Shiguang Ma$\mbox{}^\dag$ and Jie Qing$\mbox{}^\ddag$
\vspace{0.3cm} 
\rule{15cm}{1.5pt}
\end{center}


\title{}

\begin{abstract} In this paper we derive estimates for linear potentials that hold away from thin subsets. And, inspired by the celebrated 
work of Huber (cf. \cite{Cv, Hu}), we verify that, for a subset that is thin at a point, there 
is always a geodesic that reaches to the point and avoids the thin subset in general dimensions. As applications of these estimates on 
linear potentials, we consider the scalar curvature equations and slightly improve the results of Schoen-Yau \cite{SY, SYb} and Carron
\cite{Ca} on the Hausdorff dimensions of singular sets which represent the ends of complete conformal metrics on domains in manifolds of 
dimensions greater than 3. We also study $Q$-curvature equations in dimensions greater than 4 and obtain stronger
results on the Hausdorff dimensions of the singular sets (cf \cite{CHY}). More interestingly, 
our approach based on potential theory yields 
a significantly stronger finiteness theorem on the singular sets for $Q$-curvature equations in dimension 4 (cf. \cite{CQY, CH}), which
is a remarkable analogue of Huber's theorem \cite{Cv, Hu}.    
\end{abstract}

\subjclass{53C21; 31B35; 53A30; 31B05}
\keywords {Riesz potentials, Log potentials, outer capacities, $\alpha$-thinness, scalar curvature equations, $Q$-curvature equations, Hausdorff
diemsions}

\maketitle

\section{Introduction}

In this paper we employ the linear potential theory to study scalar curvature equations and $Q$-curvature equations in conformal 
geometry. This is a continuation of our recent work on $n$-superharmonic functions (cf, 
\cite{BMQ-s, BMQ-r, MQ-a, MQ-g}) inspired by the Huber's theorem and related work in superharmonic functions in dimension 2 
(cf. \cite{Cv, Hu, AH, HK}). 

Linear potential theory has always been a major subject in analysis and partial differential equations. For this paper, 
we refer readers, for instance, to \cite{Mi, AH96, AG}, for good introductions on potential theory. For clarity, the definitions of Riesz potentials and Log potentials are given in Section \ref{Sec:potential-capacity}. For our purpose, the kernel functions are not chosen
for discussions on the boundary behavior of potentials and we focus on the outer capacity and thin subsets (please see Definition 
\ref{Def:outer-capacity-euclidean} and Definition \ref{Def:alpha-thin-euclidean} in Section \ref{Sec:potential-capacity}). 
We also push to make some of the potential theory work on Riemannian manifolds. The interesting result on Riesz potentials we obtain is 

\begin{theorem}\label{Thm:intr-main-thm-1}
Suppose that $(M^n, g)$ is a complete Riemannian manifold and $\mu$ is a finite nonnegative Radon measure on a bounded 
domain $G \subset M^n$. Let $S$ be a compact subset in $G$ such that its Hausdorff dimension is greater than $d$, where 
$d < n-\alpha$ and $\alpha \in (1, n)$. Then there is a point $p\in S$ and a subset $E$ that is $\alpha$-thin at $p$ such that 
\begin{equation}\label{Equ:main-estimate-riesz-intr}
\int_{G} \frac 1{d(x, y)^{n-\alpha}} d\mu \leq \frac C{d(x, p)^{n-\alpha - d}}
\end{equation}
for some constant $C$ and all $x\in B_\delta (p) \setminus E$ for some small $\delta > 0$.
\end{theorem}  

The proof of Theorem \ref{Thm:intr-main-thm-1} 
uses a general decomposition result \cite[Proposition 1.4]{Kp} and multi-scale analysis.  We also give a proof of a slight extension of \cite[Theorem 6.3]{Mi} for Log potentials on manifolds, which is closely related to \cite{Cv, Hu, AH, MQ-a, MQ-g} for us. 
What makes these estimates useful is the following key observation about thin subsets in general dimensions (cf. \cite{Cv, Hu, AH,
MQ-a, MQ-g}).

\begin{theorem}\label{Thm:intr-main-thm-2}
Let $E$ be a subset in the Euclidean space $R^n$ and $p\in R^n$. Suppose that $E$ is $\alpha$-thin at the 
point $p$ for $\alpha \in (1, n]$. Then there is always a ray from $p$ that avoids $E$ at least within some small ball at $p$.
\end{theorem}

The proof of Theorem \ref{Thm:intr-main-thm-2} uses only the scaling property (Lemma \ref{Lem:scaling-euclidean}),
the contractive property (Lemma \ref{Lem:contractive-property}), and the calculation of $C^\alpha(S^{n-1}, B_2(0))$ 
(Lemma \ref{Lem:sphere-capacity-euclidean}) for the outer capacity $C^\alpha(E, \Omega)$ defined in 
Definition \ref{Def:outer-capacity-euclidean} and $\alpha$-thinness in Definition \ref{Def:alpha-thin-euclidean}.

In conformal geometry, the scalar curvature equation describes the conformal transformation of the scalar curvature
\begin{equation}\label{Equ:scalar-curvature-equation-intr}
- \frac {4(n-1)}{n-2}\Delta [\bar g] u + R[\bar g] u = R[u^\frac 4{n-2} \bar g]u^\frac {n+2}{n-2}.
\end{equation}
There have been many works on singular solutions after the seminal paper \cite{SY},
where the singularities represent the ends of complete conformal metrics on domains in Riemannian manifolds (cf. 
for instance, \cite[Chapter VI]{SYb} \cite{Ca} and  \cite{Sc, MS, MP}).

\begin{theorem}\label{Thm:intr-main-thm-3}
Let $(M^n, \bar g)$ be a complete Riemannian manifold and $S$ be a 
compact subset in $M^n$. And let $D$ be a bounded open neighborhood of $S$.  
Suppose that $g = u^\frac 4{n-2} \bar g$ is a conformal metric on $D\setminus S$ and is geodesically complete near $S$. 
 Then the Hausdorff dimension
\begin{equation}\label{Equ:hausdorff-dimension-intr} 
\text{dim}_\mathscr{H} (S) \leq \frac {n-2}2
\end{equation}
provided $R^-[g] \in L^\frac {2n}{n+2} (D\setminus S, g) \bigcap L^p(D\setminus S, g)$
for some $p > n/2$, where $R^-[g]$ is the negative part of the scalar curvature of the metric $g$. 
Consequently, \eqref{Equ:hausdorff-dimension-intr} holds when the scalar curvature $R[g]$ of the 
conformal metric $g$ is nonnegative.
\end{theorem}

Theorem \ref{Thm:intr-main-thm-3} is a slight improvement of  \cite[Theorem 2.7]{SY} and \cite[Theorem C]{Ca}. Our approach is
based on Theorem \ref{Thm:intr-main-thm-1} and Theorem \ref{Thm:intr-main-thm-2}. Particularly, Theorem 
\ref{Thm:intr-main-thm-3} covers domains in general manifolds, while others (cf. \cite{SY, Ca}) are restricted to domains in round 
spheres. The use of auxiliary 
testing functions built from the level sets is the key analytic technique (cf. \cite{DHM, B-V, MQ-a, MQ-g}). 
We remark that, for our approach, the complement 
$M^n\setminus D$ is not relevant (cf. Theorem \ref{Thm:main-result-R} in Section \ref{Sec:scalar-curvature}). 

In conformal geometry, one considers the Paneitz operator
$$
P_4  = \Delta^2  + \text{div} (4 A\cdot \nabla - (n-2)J\nabla ) + \frac {n-4}2 Q_4
$$ 
and the associated $Q$-curvature,
$$
Q_4 = - \Delta J + \frac n2 J^2 - 2|A|^2,
$$
where the Schouten curvature $A = \frac 1{n-2} (Ric - Jg)$ and $J = \frac 1{2(n-1)}R$.
The curvature $Q_4$, under a conformal change of the metric, transforms by the $Q$-curvature equation:
\begin{equation}\label{Equ:Q-curvature-equation}
P_4[\bar g] u= \frac {n-4}2 Q_4 [u^\frac 4{n-4} \bar g] u^\frac {n+4}{n-4} \text{ in dimensions $\geq 5$}
\end{equation}
and
\begin{equation}\label{Equ:Q-curvature-equation-4}
P_4 [\bar g] u + Q_4[\bar g] = Q_4[e^{2u}\bar g] e^{4u} \text{ in dimension $4$}.
\end{equation} 

On $Q$-curvature equations in dimensions greater than 4, we have

\begin{theorem}\label{Thm:intr-main-thm-4}
Let $(M^n, \bar g)$ be a complete Riemannian manifold for $n\geq 5$ and $S$ be a 
compact subset in $M^n$. And let $D$ be a bounded open neighborhood of $S$. Suppose that $g = u^\frac 4{n-4} \bar g$ is a 
conformal metric on $D\setminus S$ with nonnegative scalar curvature $R[g]\geq 0$ and is geodesically complete 
near $S$. And suppose also that
$$
Q_4^-[g] \in L^\frac {2n}{n+4}(D\setminus S, g),
$$
where $Q_4^-[g]$ is the negative part of $Q$-curvature of the metric $g$. Then 
\begin{equation}\label{Equ:hausdorff-dim-R-intr}
\text{dim}_{\mathscr{H}} (S) \leq \frac {n-4}2.
\end{equation}
\end{theorem}

There have been a lot of  works on the study of singular solutions to $Q$-curvature equations on manifolds of dimensions greater
than $4$,  notably \cite{QR1, QR2, CHY, GMS}, for example. Theorem \ref{Thm:intr-main-thm-4} is an improvement
of \cite[Theorem 1.2]{CHY} in terms of curvature conditions and the coverage of domains in general manifolds. 
The preliminary estimates in Lemma \ref{Lem:prelim-estimate} serve to facilitate the argument of treating 
the bi-Laplace as the iteration of the Laplace, which are interesting alternatives to usual elliptic estimates of $Q$-curvature 
equations. Again, the complement $M^n\setminus D$ is not relevant for our approach (cf. Theorem \ref{Thm:main-result-Q}
in Section \ref{Sec:Q-curvature}).

On $Q$-curvature equations in dimension 4, there have been several attempts to establish analogue results of Huber's theorem 
on finiteness of singularities (cf. \cite{CQY, CH, MQ-a, MQ-g}). $Q$-curvature in dimension 4 indeed plays a similar role as the 
Gaussian curvature does in dimension 2 (please see \eqref{Equ:Q-curvature-equation-4} for instance). 
Our following result is a significant improvement of the finiteness result of \cite[Theorem 2]{CQY} (cf. also \cite{CQY-d}). 
It covers domains in general manifolds and drops other additional curvature
assumptions in \cite[Theorem 2]{CQY}. The potential theory
approach here, particularly Theorem \ref{Thm:intr-main-thm-1} and Theorem \ref{Thm:intr-main-thm-2}, seems to be more 
effective. And the preliminary estimates in Lemma \ref{Lem:prelim-estimate-4} are interesting for $Q$-curvature
equations in dimension 4 too. Once again, the complement $M^n\setminus D$ is not relevant for our approach (cf. Theorem 
\ref{Thm:main-result-Q-4} in Section \ref{Sec:Q-curvature}).

\begin{theorem}\label{Thm:intr-main-thm-5}
Let $(M^4, \bar g)$ be a complete Riemannian manifold and $S$ be a 
compact subset in $M^n$. And let $D$ be a bounded open neighborhood of $S$. Suppose that $g= e^{2u}\bar g$ is a conformal metric on $D\setminus S$ with nonnegative scalar curvature $R[g]\geq 0$ and is geodesically complete near $S$. And suppose that
$$
\int_D Q_4^-[g]dvol[g] < \infty,
$$
where $Q_4^-[g]$ is the negative part of $Q$-curvature of the metric $g$. Then $S$ consists of only finitely many points.
\end{theorem}

The organization of this paper is as follows: In Section \ref{Sec:potential-capacity} we define linear potentials and develop potential
theory with the outer capacity and the notion of $\alpha$-thinness. Then we prove Theorem \ref{Thm:intr-main-thm-1} and Theorem 
\ref{Thm:intr-main-thm-2}. In Section \ref{Sec:scalar-curvature} we build the framework to use potential theory developed in Section 
\ref{Sec:potential-capacity} to estimate the Hausdorff dimension of singular sets which correspond to the ends of complete conformal
metrics on domains of manifolds. And we prove Theorem \ref{Thm:intr-main-thm-3}. 
In Section \ref{Sec:Q-curvature}, based on the framework built in Section \ref{Sec:scalar-curvature},
we prepare some preliminary estimates and prove Theorem \ref{Thm:intr-main-thm-4} and Theorem \ref{Thm:intr-main-thm-5}
for $Q$-curvature equations.


\section{On linear potentials} \label{Sec:potential-capacity}

The study of linear potentials has been extensive and full of great achievements. 
For this paper, readers are referred, for instance, to \cite{Mi, AH96, AG} for
good introductions. In this section we will introduce the theory of linear potential to facilitate the discussion of some
estimates of linear potential that is inspired by the one in \cite{Cv, Hu, AH, MQ-a, MQ-g}. 
The estimates provide us some alternative tools to study the problems on the
Hausdorff dimensions of singularities of solutions to a class of geometric partial differential equations in conformal geometry
(cf. \cite{SY, SYb, CHY, CH} for instance).  
We will introduce the potential theory in the way that is brief, mostly self-contained, and suffices to serve our purpose.


\subsection{Linear potential and the outer capacity in Euclidean spaces} 
For the purpose to relate potentials on Euclidean spaces to that on manifolds, we want to introduce 
potentials that are possibly confined to an open subset $\Omega \subseteq R^n$
in the Euclidean space. We will use the definition of a Radon measure on locally compact Hausdorff spaces 
in \cite[page 455]{RH}. 

\begin{definition}\label{Def:potential-euclidean}
Let $\Omega\subseteq R^n$ be a bounded open subset in the Euclidean space $R^n$. Then, for $x\in \Omega$, let 
\begin{equation}\label{Equ:potential-euclidean}
R^{\alpha, \Omega}_\mu (x) = \left\{\aligned \int_\Omega \frac 1{|x-y|^{n-\alpha}} d\mu(y)  \quad & \text{ when 
$\alpha \in (1, n)$}\\ \int_\Omega \log \frac D{|x-y|} d\mu(y) \quad & \text{ when $\alpha = n$}
\endaligned\right.
\end{equation}
for a Radon measure $\mu$ on $\Omega$, where $D$ is the diameter of $\Omega$. 
\end{definition}

For basic properties of the potential $R^{\alpha, \Omega}_\mu(x)$, 
readers are referred to \cite[Chapter 2]{Mi}. Most facts, results,  and arguments in \cite[Chapter 2]{Mi} that are 
relevant for the discussions in this paper hold with little or slight changes. 
 
\begin{definition}\label{Def:outer-capacity-euclidean} Let $E$ be a subset in $\Omega$ and $\Omega$ be a bounded
open subset in $R^n$. For $\alpha\in (1, n]$,  we define a 
capacity by
\begin{equation}\label{Equ:outer-capacity-c-euclidean}
C^\alpha (E, \Omega) = \inf \{ \mu(\Omega): \mu  \geq 0 \text{ on $\Omega$ and }
R^{\alpha, \Omega}_\mu(x) \geq 1 \text{ for all $x\in E$}\}.
\end{equation}
\end{definition}

Because of the choice of the kernel functions in Definition \ref{Def:potential-euclidean}, the capacity $C^\alpha(E, \Omega)$ 
in Definition \ref{Def:outer-capacity-euclidean} is not intended to be the same as relative capacity where the kernel function 
is the Green function for a so-called Greenian domain $\Omega$. Similar to 
\cite[Theorem 4.1 in Chapter 2]{Mi} and \cite[Section 2.6]{Mi}, we have

\begin{lemma}\label{Lem:outer-capacity-euclidean} Let $C^{\alpha}$ be the capacity defined as in 
Definition \ref{Def:outer-capacity-euclidean} for $\alpha\in (1, n]$.
\begin{enumerate}
\item $C^{\alpha}$ is nondecreasing, that is, 
$$
C^{\alpha} (E_1, \Omega) \leq C^{\alpha}(E_2, \Omega)
$$
when $E_1\subseteq E_2\subseteq\Omega\subseteq R^n$.
\item $C^{\alpha}$ is countably subadditive, that is,
$$
C^{\alpha}(\bigcup_{i=1}^\infty E_i, \Omega) \leq \sum_{i=1}^\infty C^{\alpha}(E_i, \Omega)
$$
for subsets $E_i\subseteq \Omega$.
\item $C^{\alpha}$ is an outer capacity, that is,
$$
C^\alpha(E, \Omega) = \inf \{ C^{\alpha}(U, \Omega):  E \subseteq U \text{ and  $U\subseteq \Omega$ open}\}.
$$
\end{enumerate}
\end{lemma}

The immediate and important property of the outer capacity $C^{\alpha}$ in Definition \ref{Def:outer-capacity-euclidean} 
is the scaling property (cf. \cite[page 135]{AG}).

\begin{lemma}\label{Lem:scaling-euclidean}
For a positive number $\lambda$, let
$$
A_\lambda = \{\lambda x: x\in A\}
$$
for any subset $A$ in $R^n$. Then, for $\alpha \in (1, n]$,
$$
C^{\alpha}(E_\lambda, \Omega_\lambda) = \lambda^{n-\alpha}C^\alpha(E, \Omega).
$$
\end{lemma} 

\begin{proof} For a nonnegative Radon measure $\mu$ on $\Omega$, we associate it with a nonnegative
Radon measure 
$$
\mu^* (A_\lambda) = \mu(A)
$$
on $\Omega_\lambda$. Then
$$
R^{\alpha, \Omega_\lambda}_{\mu^*} (\lambda x) = \lambda^{\alpha - n}
R^{\alpha, \Omega}_{\mu} (x)
$$
for $x\in \Omega$. Therefore
$$
\aligned
C^{\alpha} & (E_\lambda, \Omega_\lambda) = \inf \{\mu^* (\Omega_\lambda):  R^{\alpha, \Omega_\lambda}_{\mu^*} 
(\lambda x) \geq 1 \text{ for all $x\in E$}\}\\
& = \lambda^{n-\alpha} \inf\{ \lambda^{\alpha-n}\mu (E):  R^{\alpha, \Omega}_{\lambda^{\alpha -n}\mu} (x) \geq 1 
\text{ for all $x\in E$}\} \\
& = \lambda^{n-\alpha} C^\alpha(E, \Omega).
\endaligned
$$
\end{proof}

Next important property of the outer capacity $C^{\alpha}$ in Definition \ref{Def:outer-capacity-euclidean}
is the contractive property (cf. \cite{Mi, AH96, AG}).

\begin{lemma}\label{Lem:contractive-property} Suppose that 
$$
\Phi: \Omega \to \Omega
$$
is a contractive map, that is,
$$
|\Phi(x) - \Phi(y)| \leq |x - y|
$$
for all $x, y\in \Omega$. Then, for $\alpha\in (1, n]$, 
$$
C^{\alpha}(\Phi(E), \Omega) \leq C^{\alpha}(E, \Omega)
$$
for  any subset $E\subseteq \Omega$.
\end{lemma}

\begin{proof} Let $\mu$ be a nonnegative Radon measure on $\Omega$ such that $R^{\alpha, \Omega}_\mu (x) \geq 1$ for all 
$x\in E$. Then let $\mu^*$ be a nonnegative Radon measure on $\Omega$ such that $
\mu^*(A) = \mu(\Phi^{-1}(A))$ for any $A\subseteq\Omega$ and therefore
$$
\int_\Omega f(\tilde y) d\mu^*(\tilde y) = \int_\Omega f\circ\Phi(y) d\mu(y).
$$
Notice  that
$$
\aligned
R^{\alpha, \Omega}_{\mu^*} (\Phi(x)) & = \int_{\Omega} \frac 1{|\Phi(x) - \tilde y|^{n-\alpha}} d\mu^*(\tilde y)  
=  \int_{\Omega} \frac 1{|\Phi(x) - \Phi(y)|^{n-\alpha}} d\mu(y) \\
& \geq   \int_\Omega \frac 1{|x - y|^{n-\alpha}} d\mu (y) = R^{\alpha, \Omega}_{\mu} (x) \geq 1.
\endaligned
$$
Thus
$$
\aligned
C^{\alpha}(\Phi(E), \Omega) & = \inf \{\nu(\Omega): \nu  \geq 0 \text{ on $\Omega$ and }R^{\alpha, \Omega}_\nu (x) \geq 1 
\text{ for all $x\in \Phi(E)$}\}\\
& \leq  \inf \{\mu^*(\Omega): \mu^* \text{ induced from $\mu$ and } R^{\alpha, \Omega}_{\mu^*} (\Phi(x)) \geq 1 \text{ for all $x\in E$}\}\\
&  = \inf \{\mu (\Omega): \mu  \geq 0 \text{ on $\Omega$ and } R^{\alpha, \Omega}_{\mu} (x) \geq 1 \text{ for all $x\in E$}\} 
= C^\alpha(E, \Omega).
\endaligned
$$
The argument for $\alpha =n$ is similar and the proof is complete.
\end{proof}

Before we introduce the notion of thinness by $C^\alpha$, for completeness, let us calculate the outer 
capacity $C^{\alpha}(S^{n-1}, B_2)$, where 
$$
B_2 = \{x\in R^n: |x| < 2\} \text{ and } S^{n-1}= \{x\in R^n: |x| = 1\}.
$$

\begin{lemma}\label{Lem:sphere-capacity-euclidean}  (\cite[Example 5.4.3]{Mi}) For $\alpha \in (1, n]$, 
$$
C^{\alpha} (S^{n-1}, B_2) = c(n, \alpha)
$$
for some positive constant $c(n, \alpha)$.
\end{lemma}

\begin{proof} It suffices to show that $C^\alpha(S^{n-1}, B_2)$ is finite and positive. Let $\sigma$ be the volume measure for the unit sphere so that the total measure of $S^{n-1}$ is $1$. First we realize 
that the potential, for $\alpha \in (1, n]$ and $x\in S^{n-1}$,
$$
R^{\alpha, B_2}_\sigma (x)  \geq m
$$
for some $m=m(n, \alpha) >0$. 
Therefore $C^{\alpha} (S^{n-1}, B_2) \leq \frac 1m <\infty$ by Definition \ref{Def:outer-capacity-euclidean}. 
To see that $C^{\alpha} (S^{n-1}, B_2) > 0$, for 
any $\mu$ on $B_2$, we use Lemma \ref{Lem:rough-kpata} below to pick up a point $p\in S^{n-1}$ such that
\eqref{Equ:mu-B-r} holds and calculate, for $\alpha\in (1, n)$,
$$
\aligned
R^{\alpha, B_2}_\mu (p) & = (n-\alpha) \int_0^\infty \mu (\{\frac 1r - \frac 13> s\} \bigcap B_2)
\frac 1{(s+\frac 13)^{n-\alpha+1}} ds  + \frac 1{3^{n-\alpha}} \mu(B_2)\\
& = (n-\alpha)\int_0^3 \mu(B_r(p)\bigcap B_2) r^{\alpha - n -1} dr + \frac 1{3^{n-\alpha}} \mu(B_2)\\ 
&\leq M (n, \alpha) \mu(B_2)
\endaligned
$$
for some $M(n, \alpha) >0$ and $r=|x-p|$. For $\alpha = n$,
$$
\aligned
R^{\alpha, B_2}_\mu (p) & =  \int_0^\infty \mu( \{\frac 3r - 1 > s\} \bigcap B_2) \frac 1{1+s} ds + \log\frac 43\mu(B_2) \\
& = \int_0^3 \mu(B_{r}(p)\bigcap B_2) \frac 1r dr + \log\frac 43\, \mu(B_2) \\
& \leq M(n, n) \mu(B_2)
\endaligned
$$
for some $M(n, n) >0$. In the above we used \cite[Theorem 8.16]{Ru}.
This implies $C^{\alpha} (S^{n-1}, B_2) \geq \frac 1{M(n, \alpha)} >0$ by 
Definition \ref{Def:outer-capacity-euclidean}. Thus the proof is complete.
\end{proof}

By the Vitali covering lemma, we prove the following fact used in the above.

\begin{lemma}\label{Lem:rough-kpata} Let $n\geq 2$ and $\mu$ be a finite nonnegative Radon measure on 
$B_2\subset R^n$. Then there is a point $p\in S^{n-1}$ such that
\begin{equation}\label{Equ:mu-B-r}
\mu(B_r(p)\bigcap B_2) \leq c(n) \mu(B_2) r^{n-1} \text{ for all $r>0$}.
\end{equation}
for some dimensional constant $c=c(n)$.
\end{lemma}

\begin{proof} For convenience, let $\mu(B_2)=1$. Assume otherwise, for any $q\in S^{n-1}$, there is
$r_q>0$ such that
$$
\mu(B_{r_q}(q)\bigcap B_2) \geq  c(n) r_q^{n-1}.
$$
Using Vitali covering lemma, we have $\{q_1, q_2, \cdots, q_k\}\subset S^{n-1}$ such that the collection of balls
$$
\{B_{r_{q_1}}(q_1), B_{r_{q_2}}(q_2), \cdots B_{r_{q_k}}(q_k)\}
$$
are disjoint but the collection of balls 
$$
\{B_{3r_{q_1}}(q_1), B_{3r_{q_2}}(q_2), \cdots B_{3r_{q_k}}(q_k)\}
$$
cover the sphere $S^{n-1}$.  Therefore, on one hand, 
$$
c(n) \sum_{i=1}^k r_{q_i}^{n-1} \leq \sum_{i=1}^k \mu(B_{r_{q_i}}(q_i)\bigcap B_2) \leq \mu(B_2) = 1.
$$
On the other hand, 
$$
|S^{n-1}| \leq \sum_{i=1}^k |B_{3r_{q_i}}(q_i)\bigcap S^{n-1}| <|S^{n-1}| c(n) \sum_{i=1}^k r_{q_i}^{n-1}
$$
when $c(n)$ is sufficiently large, where $|\cdot|$ stands for the Lebesgue measure on $S^{n-1}$. 
Therefore the lemma is proven by contradiction.
\end{proof}

Now let us introduce the geometric definition of thinness. For notions of thinness in terms of the fine topology
and Wiener criterion, readers are referred, for instance, to \cite{Mi, AH96, AG}. Let  
$$
\omega^\delta_i (p) = \{ x\in R^n: |x - p| \in [2^{-i}\delta,  2^{-i+1}\delta]\}
$$
and
$$
\Omega^\delta_i(p) = \{x\in R^n: |x - p| \in (2^{-i-1}\delta , 2^{-i+ 2} \delta)\}.
$$

\begin{definition} \label{Def:alpha-thin-euclidean} Let $E$ be a subset in the Euclidean space $R^n$ and $p\in R^n$ be a point
in $R^n$. The subset $E$ is said to be $\alpha$-thin at the point $p$ for $\alpha \in (1, n)$ if 
$$
\sum_{i\geq 1} \frac { C^{\alpha} (E\bigcap \omega^\delta_i (p), \Omega^\delta_i (p))}
{C^{\alpha} (\partial B_{2^{-i}\delta} (p), B_{2^{-i+1}\delta}(p))}
< \infty
$$  
for some small $\delta>0$.  The subset $E$ is said to be $n$-thin at $p$ if 
$$
\sum_{i\geq 1} i C^{n} (E\bigcap \omega^\delta_i(p), \Omega^\delta_i(p)) < \infty
$$
for some small $\delta > 0$.
\end{definition}

Combining Lemma \ref{Lem:outer-capacity-euclidean}-\ref{Lem:sphere-capacity-euclidean} with the above definition, we observe the following important property of $\alpha$-thin sets, inspired by \cite{AH} (see also \cite{MQ-a, MQ-g}). We recall 
Theorem \ref{Thm:intr-main-thm-2} from the introduction for readers' convenience.

\begin{theorem}\label{Thm:segment-euclidean} 
Let $E$ be a subset in the Euclidean space $R^n$ and $p\in R^n$ be a point. Suppose that $E$ is $\alpha$-thin at the 
point $p$ for $\alpha \in (1, n]$. Then there is a ray from $p$ that avoids $E$ at least within some small ball at $p$.
\end{theorem}

\begin{proof}  First of all, due to the translation invariance, we may simply assume $p$ is the origin of the Euclidean space. 
Then, by the scaling property of the outer capacity $C^\alpha$ in Lemma \ref{Lem:scaling-euclidean}, one 
notices that
$$
 \frac { C^{\alpha } (E\bigcap \omega^\delta_i, \Omega^\delta_i)}
{C^{\alpha} (\partial B_{2^{-i}\delta} , B_{2^{-i+1}\delta})}
 =  \frac { C^{\alpha} (S_i(E)\bigcap \omega^1_0, \Omega^1_0)}{C^{\alpha} (\partial B_1, B_2)}
$$
where the scaling map: $S_i (v) = \frac {2^{i}}\delta v$. Then we consider the projection
$$
P(v) = \left\{\aligned \frac v{|v|} \quad & \text{ when $v\in R^n$ and $|v|\geq 1$}\\
v \quad & \text{ when $v\in R^n$ and $|v| < 1$}\endaligned\right.,
$$
which is contractive. Therefore, in the light of Lemma \ref{Lem:contractive-property}, we have
$$
C^{\alpha} (P(S_i(E)\cap \omega^1_0), \Omega^1_0) \leq C^{\alpha} (S_i(E)\cap \omega^1_0, \Omega^1_0).
$$
Next, using the countable sub-additivity in Lemma \ref{Lem:outer-capacity-euclidean}, we have
$$
C^{\alpha} (\bigcup_{i\geq k} P(S_i(E)\cap \omega^1_0), \Omega^1_0)  \leq \sum_{i\geq k} 
C^{\alpha} (P(S_i(E)\cap \omega^1_0), \Omega^1_0).
$$
Thus, 
$$
\aligned
C^{\alpha} & (\bigcup_{i\geq k}  P(S_i(E)\cap \omega^1_0), \Omega^1_0)  \leq \sum_{i\geq k} 
C^{\alpha} (S_i(E)\cap\omega^1_0, \Omega^1_0)\\
& \leq C^{\alpha} (\partial B_1,  B_2) 
\sum_{i\geq k}  \frac { C^{\alpha} (S_i(E)\bigcap \omega^1_0, \Omega^1_0)}{C^{\alpha} (\partial B_1, B_2)}\\
& \leq C^{\alpha} (\partial B_1, B_2) 
\sum_{i\geq k}  \frac { C^{\alpha} (E \bigcap \omega^\delta_{i}, \Omega^\delta_{i})}{C^{\alpha}
( \partial B_{2^{-i}\delta}, B_{2^{-i+1}\delta})}
\endaligned
$$
which is arbitrarily small when $k$ is appropriately large using Lemma \ref{Lem:sphere-capacity-euclidean} 
for $C^{\alpha}(\partial B_1, B_2)$. And then this implies that
$$
\partial B_1 \setminus \bigcup_{i\geq k}  P(S_i(E)\cap \omega^1_0)\neq \emptyset.
$$
The argument for $\alpha = n$ is similar and easier. And the proof is complete.
\end{proof}


\subsection{Linear potential on manifolds}
On a given complete Riemannian manifold $(M^n, g)$, let $d(\cdot, \cdot)$ 
be the distance function associated with the given Riemannian metric $g$. 

\begin{definition} Suppose that $(M^n, g)$ is a complete Riemannian manifold and $U\subseteq M^n$ is a bounded 
open subset. For $\alpha\in (1, n]$, the linear potential on the Riemannian manifold $(M^n, g)$ of order $\alpha$ for a 
Radon measure $\mu$ on $U$ is given as
$$
\mathscr{R}^{\alpha, U}_\mu (x) = \left\{\aligned \int_U \frac 1{d(x, y)^{n-\alpha}} d\mu(y)  \quad & \text{ when 
$\alpha \in (1, n)$}\\ \int_\Omega \log \frac D{d(x, y)} d\mu(y) \quad & \text{ when $\alpha = n$}
\endaligned\right.,
$$
where $D$ is the diameter of $U$.
\end{definition} 

From the discussion in the previous subsection, it is easily seen that one may generate an outer capacity 
$\mathscr{C}^\alpha(E, U)$ for any subset $E\subseteq U\subseteq M^n$ that behaves like the counter part 
in Euclidean spaces. To use $R^{\alpha, \Omega}_\mu(x)$ and $C^{\alpha}(E, \Omega)$ on Euclidean spaces 
in the previous subsection to study $\mathscr{R}^{\alpha, U}_\mu(p)$ and $\mathscr{C}^\alpha(A, U)$ on manifolds, 
we first introduce the correspondence between Radon measures on the tangent space $T_pM^n$ at each point 
$p\in M^n$ and those on $(M^n, g)$. Suppose that $(M^n, g)$ is a complete Riemannian manifold. 
Let $p\in M^n$ and $U$ be a convex normal coordinate neighborhood at $p$ on $(M^n, g)$, where
the exponential map serves as the convex normal coordinate 
$$
\text{exp}|_p: \Omega \to U.
$$
The domain $U$ is said to be convex if the unique geodesic joining any two points in $U$ stays in $U$. Moreover, we 
may assume in the coordinate chart $U$ the exponential map be uniformly bi-Lipschitz throughout this paper.

Then, for a Radon measure $\mu$ on $U\subseteq M^n$, one may introduce the Radon measure 
$\mu^*$ on $\Omega\subset T_pM^n$ such that, for a subset $E \subseteq \Omega$, 
$$
\mu^* (E) = \mu (\text{exp}|_p E) \text{ and }  \int_\Omega f\circ\text{exp}|_p d\mu^* = \int_U f d\mu.
$$
It is then easily seen that the following equivalence between the linear potential $R^{\alpha, \Omega}_{\mu^*}$,
the outer capacities $C^{\alpha}(\cdot, \Omega)$ and the corresponding $\mathscr{R}^{\alpha, U}_{\mu}$,
$\mathscr{C}^{\alpha}(\cdot, U)$ holds. Namely,

\begin{lemma}\label{Lem:exp-T_p} Suppose that $(M^n, g)$ be a complete Riemannian manifold and $p\in M^n$. Let 
$$
\text{exp}|_p: \Omega \to U
$$
be the convex normal coordinate chart, where the exponential map is uniformly bi-Lipschitz. And let 
$\alpha \in (1, n]$. Then, for 
$A\subset U$ and $E = (\text{exp}|_p)^{-1}A \subset \Omega$,
$$
\aligned 
C^{-1} R^{\alpha, \Omega}_{\mu^*} \leq  & \mathscr{R}^{\alpha, U}_\mu \leq C R^{\alpha, \Omega}_{\mu^*}\\
C^{-1} C^\alpha(E, \Omega) \leq & \mathscr{C}^{\alpha} (A, U) \leq C C^\alpha(E, \Omega)
\endaligned
$$
for some constant $C= C(M^n, g, U, p)$. Consequently, a subset $A\subset U$ is $\alpha$-thin at $p$ if and only if
$E= (\text{exp}|_p)^{-1}(A)\subset\Omega$ is $\alpha$-thin at the origin of $T_pM^n$. 
\end{lemma} 

\begin{proof} The proof is straightforward based on the properties of the convex normal coordinate chart at a point in a 
complete Riemannain manifold, where the exponential map is bi-Lipschtiz. 
\end{proof}


\subsection{Estimates of Riesz potential} In this subsection, we introduce our estimates of Riesz potentials on manifolds. 
We will recall some well known estimates for Riesz potential in Euclidean spaces \cite[Chapter 2]{Mi}.  

Our estimates on Riesz potentials are designed to help understand the Hausdorff dimensions of singularities of solutions of partial 
differential equations on manifolds. Let us start with a general decomposition theorem for nonnegative Radon measures on a complete Riemannian manifold based on \cite[Proposition 1.4]{Kp}, which is related to Lemma \ref{Lem:rough-kpata} and
a broad generalization of the Lebesgue Differentiation Theorem in some way.

\begin{lemma} \label{Lem:lebesgue} ( \cite[Proposition 1.4]{Kp}) 
Let $\mu$ be a nonnegative Radon measure on a complete Riemannian manifold $(M^n, g)$ and let
$$
G_d^\infty = \{x\in M^n: \limsup_{r\to 0} r^{-d} \mu(B_r(x)) = +\infty\}
$$
for any $d\in [0, n]$.
Then 
$$
\mathscr {H}_d (G^\infty_d) = 0
$$
where $\mathscr{H}_d$ is the Hausdorff measure of dimension $d$.
\end{lemma}

\begin{proof} Based on the general decomposition theorem \cite[Proposition 1.4]{Kp} on the Euclidean space and  the 
correspondence of Radon measures in Lemma \ref{Lem:exp-T_p}, this lemma is easily seen. Specifically, we first 
prove the statement for Radon measures supported in a convex normal coordinate chart used in 
Lemma \ref{Lem:exp-T_p}. Then the lemma follows by using a countable covering for $(M, g)$ by convex normal 
coordinate charts.
\end{proof} 

Now we are ready to state and prove one crucial analytic result in this paper on the behavior of the Riesz potentials.
For readers' convenience, we recall Theorem \ref{Thm:intr-main-thm-1} from the introduction.

\begin{theorem} \label{Thm:main-estimate}
Suppose that $(M^n, g)$ is a complete Riemannian manifold and $\mu$ is a finite Radon measure on a bounded 
domain $G \subset M^n$. Let $S$ be a compact subset in $G$ such that its Hausdorff dimension is greater than $d$. 
And let $\alpha \in (1, n)$ and 
$d < n-\alpha$. Then there is a point $p\in S$ and a subset $E$ that is $\alpha$-thin at $p$ such that 
$$
\int_{G} \frac 1{d(x, y)^{n-\alpha}} d\mu \leq \frac C{d(x, p)^{n-\alpha - d}}
$$
for some constant $C$ and all $x\in B_\delta (p) \setminus E$ for some $\delta > 0$.
\end{theorem}

\begin{proof} First, due to the assumption that the Hausdorff dimension of $S$ is greater than $d$, 
$$
\mathscr{H}_{d+\epsilon}(S) = \infty
$$
for some small $\epsilon >0$. Then, in the light of Lemma \ref{Lem:lebesgue}, there is a point $p\in S$ such that
$$
\limsup_{r\to 0} r^{-(d+\epsilon)} \mu(B_r(p)) \leq C < \infty.
$$
That is to say
\begin{equation}\label{Equ:mu-on-ball}
\mu(B_r(p)) \leq C r^{d+\epsilon}
\end{equation}
when $r$ is appropriately small. Secondly, we may confine ourselves to a convex normal coordinate neighborhood $U$ of $p$ and 
we may work on the Euclidean space without the loss of generality in the light of the discussion in the previous subsection, particularly, Lemma \ref{Lem:exp-T_p}, where $\text{exp}|_p: \Omega\to U$ and $\text{exp}|_p (0) = p$. For convenience, we will not 
differentiate $\mu$ and $\mu^*$ if no confusion rises. 
Therefore, for $x\in \omega_i^\delta\subset \Omega$ when $\delta$ is sufficiently small and $i$ is appropriately large, 
\begin{equation}\label{Equ:calculation-1}
\aligned
& R^{\alpha, \Omega}_\mu (x)  = \int_\Omega \frac 1{|x - y|^{n-\alpha}} d\mu\\
& = \int_{\Omega\setminus B_{2^{-i_0+2}\delta} } \frac 1{|x - y|^{n-\alpha}} d\mu + \int_{B_{2^{-i_0+2}\delta} \setminus 
\Omega^\delta_i}  \frac 1{|x - y|^{n-\alpha}} d\mu + \int_{\Omega^\delta_i} \frac 1{|x - y|^{n-\alpha}} d\mu,
\endaligned
\end{equation}
where $i_0 \leq i$ to be fixed. For the first term in the right hand side of \eqref{Equ:calculation-1}, 
$$
I = \int_{\Omega\setminus B_{2^{-i_0+2}\delta}} \frac 1{|x - y|^{n-\alpha}} d\mu \leq
(\frac 1{(2^{-i_0+2}\delta - 2^{-i+1}\delta})^{n-\alpha} \mu(\Omega)
\leq (\frac 1{2^{-i_0+1}\delta})^{n-\alpha} \mu(\Omega).
$$
Recall that $2^{-i}\delta \leq |x| \leq 2^{-i +1}\delta$ for $x\in \omega_i$, we have
\begin{equation}\label{Equ:term-I}
I  \leq \mu(\Omega) \frac {(2^{-i+1}\delta)^{n-\alpha-d}}{(2^{-i_0+1}\delta)^{n-\alpha}} \frac 1{|x|^{n-\alpha - d}}
\leq C \frac 1{|x|^{n-\alpha-d}}
\end{equation}
where $C = C(n, \alpha, d, \delta, i_0)$. For the second term in the right hand side of \eqref{Equ:calculation-1},
$$
\aligned
& \int_{B_{2^{-i_0+2}\delta} \setminus \Omega_i^\delta} \frac 1{|x - y|^{n-\alpha}} d\mu = 
\int_{B_{2^{-i_0+2}\delta} \setminus B_{2^{-i+2}\delta}}  \frac 1{|x - y|^{n-\alpha}} d\mu +
\int_{B_{2^{-i - 1}\delta}}  \frac 1{|x - y|^{n-\alpha}} d\mu \\
& \leq \int_{B_{2^{-i_0+2}\delta} \setminus B_{2^{-i+2}\delta}}  \frac 1{|x - y|^{n-\alpha}} d\mu +
(\frac 1{2^{-i-1}\delta})^{n-\alpha}\mu(B_{2^{-i-1}\delta}) \\
& \leq \sum_{k=i_0}^{i-1} \int_{B_{2^{-k+2}}\delta \setminus B_{2^{-k+1}\delta}}  \frac 1{|x - y|^{n-\alpha}} d\mu 
+ (\frac 1{2^{-i-1}\delta})^{n-\alpha} \mu(B_{2^{-i-1}\delta})\\
& \leq \sum_{k=i_0}^{i-1} (\frac 1{2^{-k}\delta})^{n-\alpha} \mu(B_{2^{-k+2}\delta}) 
+ (\frac 1{2^{-i-1}\delta})^{n-\alpha} \mu(B_{2^{-i-1}\delta}).
\endaligned
$$
Using \eqref{Equ:mu-on-ball} for $\epsilon = 0$, we continue from the above,
\begin{equation}\label{Equ:term-II}
\aligned
II & \leq  C  (4^{d} \sum_{k=i_0}^{i-1} (\frac 1{2^{-k}\delta})^{n-\alpha - d}
+ (\frac 1{2^{-i-1}\delta})^{n-\alpha - d})\\
& \leq C(\frac {4^{d}}{1 - 2^{-(n-\alpha - d)}} (\frac 1{2^{-i + 1}\delta})^{n-\alpha-d}
+  (\frac 1{2^{-i-1}\delta})^{n-\alpha - d})\\
& \leq C \frac 1{|x|^{n-\alpha - d}},
\endaligned
\end{equation}
where $C = C(n, \alpha, d, \delta, i_0)$. To handle the third term in the right hand side of \eqref{Equ:calculation-1}, we let
$$
E_i^\lambda = \{x\in \omega^\delta_i: \int_{\Omega^\delta_i} \frac 1{|x - y|^{n-\alpha}} d\mu \geq \lambda 2^{i(n-\alpha - d)}\},
$$
where $\lambda >0$ to be fixed. By Definition \ref{Def:outer-capacity-euclidean}, we know
$$
C^\alpha (E^\lambda_i, \Omega^\delta_i) \leq \frac {\mu(\Omega^\delta_i)}{\lambda 2^{i(n-\alpha-d)}}
\leq \frac  C\lambda \frac {(2^{-i+2}\delta)^{d+\epsilon}}{2^{i(n-\alpha - d)}} = \frac {C4^{d+\epsilon}}\lambda 2^{-i\epsilon}
(2^{-i})^{n-\alpha},
$$
where \eqref{Equ:mu-on-ball} for some $\epsilon > 0$ is used and $\Omega^\delta_i \subset B_{2^{-i+2}\delta}$. 
Now, from Lemma \ref{Lem:sphere-capacity-euclidean} and the scaling property, we know
$$
C^\alpha (\partial B_{2^{-i}\delta}, B_{2^{-i +1}\delta}) = C(n, \alpha) (2^{-i}\delta)^{n-\alpha}
$$
and
$$
\sum_{i\geq i_0} \frac {C^\alpha(E^\lambda_i, \Omega^\delta_i)}{C^\alpha(\partial B_{2^{-i}\delta}, B_{2^{-i+1}\delta})}
\leq \frac C\lambda \sum_{i\geq i_0} 2^{-\epsilon i} <\infty.
$$
Thus, by Definition \ref{Def:alpha-thin-euclidean}, the proof is completed.
\end{proof}

As a consequence of Theorem \ref{Thm:segment-euclidean} and \ref{Thm:main-estimate}, we have

\begin{corollary} Suppose that $(M^n, g)$ is a complete Riemannian manifold and $\mu$ is a finite 
Radon measure on a bounded 
domain $G \subset M^n$. Let $S$ be a compact subset in $G$ such that its Hausdorff dimension is greater than $d$. 
And let $\alpha \in (1, n)$ and 
$d < n-\alpha$. Then there is a point $p\in S$ and a subset $E$ such that 
$$
\int_{G} \frac 1{d(x, y)^{n-\alpha}} d\mu \leq \frac C{d(x, p)^{n-\alpha - d}}
$$
for some constant $C$ and for all $x$ along a geodesic ray from $p$ at least within a small geodesic ball. 
\end{corollary}


\subsection{Estimates of Log potential}  First, as stated in \cite[Theorem 6.3]{Mi}, for the Log potential $U_n\mu (x)$ 
on Euclidean spaces defined in \cite[page 82]{Mi}, 
$$
\lim_{x\to p \text{ and } x\in \Omega\setminus E} \frac {U_n\mu(x)}{\log \frac 1{|x-p|}} = \mu(\{p\}).
$$

The following is our version of \cite[Theorem 6.3]{Mi} on manifolds. For us it is a generalization of \cite[Theorem 1.3]{AH} in higher dimensions and linear version of such behaviors for $n$-superharmonic functions (cf. \cite{Hu, BMQ-s, BMQ-r, MQ-a, MQ-g}). For convenience, we present a brief but full proof based on the potential theory
developed in previous subsections in this paper. 

\begin{theorem}\label{Thm:n-potential} Suppose $(M^n, g)$ is a complete Riemannian manifold.   
Let $\mu$ be a finite Radon measure on a bounded domain $G\subset M^n$.  
Then, for $p\in G$,  there is a subset $A$ that is n-thin at $p$ and
$$
\lim_{x \to p \text{ and } x \in M^n\setminus A} \frac {\int_{G} \log \frac 1{d(x, p)} d\mu(x)}{\log \frac 1{d(x, p)}} = \mu(\{p\}).
$$
\end{theorem}

\begin{proof} Let 
$$
\text{exp}|_p : \Omega\to U
$$
be a convex normal coordinate at $p\in M^n$. Clearly, it suffices to show that, there is a subset $A$ in $U$, which is 
$n$-thin at $p$, such that 
\begin{equation}\label{Equ:log-potential-fine}
\lim_{x \to p \text{ and } x \in U \setminus A} \frac {\mathscr{R}^{n, U}_\mu (x)}{\log \frac 1{d(x, p)}} = \mu(\{p\}).
\end{equation}
Therefore, for $x\in \omega^\delta_i (p)$, we write
\begin{equation}\label{Equ:calculation-2}
\aligned
& \mathscr{R}^{n, U}_\mu (x)  = \int_U \log \frac D{d(x, y)} d\mu (y)\\
& = \int_{U \setminus B_{2^{-i_0+2}\delta} } \log \frac D{d(x, y)} d\mu  + \int_{B_{2^{-i_0+2}\delta} \setminus \Omega^\delta_i} 
\log \frac D{d(x, y)} d\mu  + \int_{\Omega^\delta_i} \log \frac D{d(x, y)} d\mu.
\endaligned
\end{equation}
Here we omit the center $p$ for each ball or annulus for simplicity. For the first term in the right hand side of \eqref{Equ:calculation-2}, 
\begin{equation}\label{Equ:term-I-2}
I = \int_{U\setminus B_{2^{-i_0+2}\delta}} \log\frac D{d(x, y)} d\mu \leq
 \mu(U) \log \frac D{2^{-i_0 +1}\delta} = o(1)\log\frac 1{d(x, p)} \text{ as $x\to p$}.
\end{equation}
For the second term in the right hand side of \eqref{Equ:calculation-2},
$$
\aligned
& \int_{B_{2^{-i_0+2}\delta} \setminus \Omega_i^\delta} \log\frac D{d(x, y)} d\mu (y) = 
\int_{B_{2^{-i_0+2}\delta} \setminus B_{2^{-i+1}\delta}} \log\frac D{d(x, y)} d\mu +
\int_{B_{2^{-i - 2}\delta}} \log\frac D{d(x, y)} d\mu \\
& \leq C [\sum_{k=i_0}^{i} k \mu(B_{2^{-k+2}\delta} \setminus B_{2^{-k +1}\delta}) ]
+ \mu(B_{2^{-i-2}\delta}) \log\frac D{2^{-i -2}\delta}.
\endaligned
$$
Due to the regularity of Radon measures and $d(x, p)\in [2^{-i-1}\delta, 2^{-i}\delta]$, we know
\begin{equation}\label{Equ:II-1}
\mu(B_{2^{-i-2}\delta}) \log\frac D{2^{-i -2}\delta} = \mu(\{p\}) \log \frac 1{d(x, p)} + o(\log\frac 1{d(x, p)}) \text{ as $x\to p$}
\end{equation}
and 
\begin{equation}\label{Equ:II-2}
\sum_{k=i_0}^{i} k \mu(B_{2^{-k+2}\delta} \setminus B_{2^{-k +1}\delta}) = o(1) i =
o(1) \log\frac 1{d(x, p)}
\end{equation}
as $i\to\infty$ or equivalently $x\to p$. To see \eqref{Equ:II-2}, for any $\epsilon > 0$, we first find $k_0$ such that
$$
\mu(B_{2^{-l+2}\delta}\setminus B_{2^{-m+1}\delta}) \leq \frac 12 \epsilon
$$
for all $m\geq l\geq k_0$ due to the regularity of $\mu$. Next, we find $N$ such that
$$
\frac{\sum_{k=i_0}^{k_0} k\mu(B_{2^{-k+2}\delta}\setminus B_{2^{-k+1}\delta})}i \leq \frac 12\epsilon
$$
for all $i\geq N$. Together, this gives
$$
\aligned
& \frac {\sum_{k=i_0}^{i} k \mu(B_{2^{-k+2}\delta} \setminus B_{2^{-k +1}\delta}) }i \\
& = 
\frac {\sum_{k=i_0}^{k_0} k \mu(B_{2^{-k+2}\delta} \setminus B_{2^{-k +1}\delta}) }i + 
\frac {\sum_{k=k_0+1}^{i} k \mu(B_{2^{-k+2}\delta} \setminus B_{2^{-k +1}\delta})} i \\
& = 
\frac {\sum_{k=i_0}^{k_0} k \mu(B_{2^{-k+2}\delta} \setminus B_{2^{-k +1}\delta}) }i + 
\sum_{k=k_0+1}^{i}  \mu(B_{2^{-k+2}\delta} \setminus B_{2^{-k +1}\delta})\\
& \leq \epsilon
\endaligned
$$
for all $i\geq N$.
Thus we conclude that
\begin{equation}\label{Equ:term-II-2}
II = (\mu(\{p\}) +o(1)) \log\frac 1{d(x, p)} \text{as $x\to p$}.
\end{equation}
To handle the third term in the right side of \eqref{Equ:calculation-2}, for $\lambda_i>0$ to be determined, we consider 
$$
A^{\lambda_i} = \{x\in \omega^\delta_i: \int_{\Omega^\delta_i} \log\frac {D_i}{d(x, y)}d\mu \geq i\lambda_i \},
$$
where $D_i$ is the dimeter of $\Omega_\delta^i$. By Definition \ref{Def:outer-capacity-euclidean}, 
$$
\mathscr{C}^n(A^{\lambda_i}, \Omega^\delta_i) \leq \frac {\mu(\Omega^\delta_i)}{i\lambda_i}.
$$
In the light of Definition \ref{Def:alpha-thin-euclidean}, we consider
$$
\sum_{i\geq i_0} i \mathscr{C}^n (A^{\lambda_i}, \Omega^\delta_i) \leq \sum_{i\geq i_0} \frac {\mu(\Omega^\delta_i)}{\lambda_i}
$$
and pick up $\lambda_i\to 0$ as $i\to\infty$  by the classic Paul du Bois-Reymond Theorem (cf. \cite[(5) Page 40]{Brom} and
\cite{Rey}) for infinite series such that $\sum_{i\geq i_0} \frac {\mu(\Omega^\delta_i)}{\lambda_i}$ converges when 
$\sum_{i\geq i_0} \mu(\Omega^\delta_i)$ converges. This is to say that the third term in the right side of \eqref{Equ:calculation-2}
\begin{equation}\label{Equ:term-III-2}
\aligned
III & = \int_{\Omega^\delta_i} \log\frac D{d(x, y)}d\mu(y) = \int_{\Omega^\delta_i} \log\frac {D_i}{d(x, y)}d\mu(y)  
+ \log \frac D{D_i} \mu(\Omega^i_\delta)\\
&  \leq (\lambda_i+ (1 + \frac 1i \log\frac 1\delta)\mu(\Omega^i_\delta)) \log\frac D{d(x, p)}\\
& = o(1) \log\frac 1{d(x, p)} \text{ as $x\in\omega^\delta_i\setminus E^{\lambda_i}$ and $x\to p$}.
\endaligned
\end{equation}
Finally, if let $A=\bigcup_i A^{\lambda_i}$, we have
$$
\lim_{x \to p \text{ and } x \in U\setminus A} \frac {\mathscr{R}^{n, U}_\mu (x)}{\log \frac 1{d(x, p)}} = \mu(\{p\}),
$$
where $A$ is $n$-thin at $p$. The proof is complete.
\end{proof}


\section{On scalar curvature equations}\label{Sec:scalar-curvature}

In this section we focus on the scalar curvature equations for conformal deformation of metrics. Let $(M^n, \bar g)$ be
a compact Riemannian manifold for $n\geq 3$. 
Let $R_{ijkl}[\bar g]$ be the Riemann curvature tensor, $R_{ij} [\bar g]= R_{ijkl}\bar g^{kl}$ be 
the Ricci curvature tensor, and $R [\bar g] = R_{ij}\bar g^{ij}$ be the scalar curvature. 
The scalar curvature equation in conformal geometry is
\begin{equation}\label{Equ:yamabe-equation}
- \frac {4(n-1)}{n-2} \Delta[\bar g] u+ R[\bar g]u = R[u^\frac 4{n-2} \bar g] u^\frac {n+2}{n-2}
\end{equation}
for a positive function $u$. The scalar curvature equation describes how the scalar curvature transforms under 
conformal change of metrics. 
In this section we want to use the estimates for the Newton potential in the previous section to study the Hausdorff 
dimensions of the singularities of solutions $u$ to the scalar equations which represent the ends of a complete
conformal metric $u^\frac 4{n-2} \bar g$. 

We remark here that all of the results in this section hold if we assume $S$ is compact, $D\subset M^n$ is a bounded domain that
contains $S$, and $(M^n, \bar g)$ is just complete, because the possible noncompact part $M^n\setminus \bar D$ is not relevant 
for the purpose here.


\subsection{Preliminaries} Let us start with  \cite[Lemma 3.1]{MQ-g}, which
is a slight improvement of \cite[Proposition 8.1]{CHY}. 

\begin{lemma}\label{Lem:lemma-3.1} ( \cite[Lemma 3.1]{MQ-g}) 
Let $(M^n, \bar g)$ be a compact Riemannian manifold and $S$ be a closed subset in $M^n$. And let
$D$ be an open neighborhood of $S$.
Suppose that $g = u^\frac 4{n-2} \bar g$ is a conformal metric on $D\setminus S$ and is geodesically complete near $S$. Then 
$$
u(x) \to +\infty \text{ as $x\to S$}
$$
if 
$R^-[g] \in L^p (D\setminus S, g)$ for some $p > n/2$, where $R^- [g] = \max\{ - R[g], 0\}$ 
stands for the negative part of the scalar curvature $R[g]$ and $L^p(D\setminus S, g)$ 
is the $L^p$ space with respect to the metric $g$.
\end{lemma}
 
For a preliminary estimate on the Hausdorff dimension of $S$, we follow the proof of \cite[Theorem 3.1]{MQ-g} and get 

\begin{proposition}\label{Prop:newton-capacity} Let $(M^n, \bar g)$ be a compact Riemannian manifold and $S$ be a 
closed subset in $M^n$. And let $D$ be an open neighborhood of $S$ where the scalar curvature $R[\bar g]$ is nonpositive.  
Suppose that $g = u^\frac 4{n-2} \bar g$ is a conformal metric on $D\setminus S$ and is geodesically complete near $S$. 
Then the Newton capacity of $S$ is zero and therefore the Hausdorff dimension
$$
\text{dim}_{\mathscr{H}} (S)\leq n-2,
$$  
provided that $$R^-[g] \in L^\frac {2n}{n+2}(D\setminus S, g) \bigcap L^p(D\setminus S, g)$$ for some $p> n/2$.
\end{proposition}

\begin{proof} Recall the scalar curvature equation
\begin{equation}\label{Equ:right-yamabe-equation}
-\frac {4(n-1)}{n-2}\Delta u = -R u + R^+[g] u^\frac {n+2}{n-2} - R^-[g] u^\frac {n+2}{n-2} \text{ in $D\setminus S$},
\end{equation}
where
\begin{equation}\label{Equ:L^1-R^-}
\aligned
\int_{D\setminus S} R^-[g] u^\frac {n+2}{n-2} dvol & \leq (\int_{D\setminus S} (R^-[g])^\frac {2n}{n+2} u^\frac{2n}{n-2} 
dvol)^\frac {n+2}{2n} \text{vol}(D)^\frac {n-2}{2n} \\
& \leq (\int_{D\setminus S} (R^-[g])^\frac {2n}{n+2} dvol[g])^\frac {n+2}{2n} \text{vol}(D)^\frac {n-2}{2n}\\
& < \infty.
\endaligned
\end{equation}
Here, and from now on in the following, all geometric quantities are under the background metric $\bar g$ unless 
indicated otherwise. And, in the light of Lemma \ref{Lem:lemma-3.1}, we know
$$
u(x) \to +\infty \text{ as $x\to S$}.
$$
As in the proof of \cite[Theorem 3.1]{MQ-g} (adopted from \cite[Lemma 1.2]{B-V}), we use the following test functions. First we let
$$
u_{\alpha, \beta} = \left\{\aligned \beta & \quad u \geq \alpha + \beta;\\
u-\alpha & \quad u < \alpha + \beta.\endaligned\right. \text{ and } \phi_{\alpha, \beta} = u_{\alpha, \beta} - \beta + \beta (1 - \eta),
$$
where $\eta \in C^\infty_c(\Sigma_\alpha)$ is a fixed cut-off function that is equal to one in a neighborhood of $S$ 
and $\Sigma_\alpha = \{x\in D: u(x) > \alpha\}$. Notice that, for $\beta$ sufficiently large,
$$
u_{\alpha, \beta} \in (0, \beta] \text{ in $\Sigma_\alpha$ and } \phi_{\alpha, \beta} = 0 \text{ on $\{x\in D: u(x) = \alpha\}
\bigcup \{x\in D: u\geq \alpha + \beta$}\}
$$
and
$$
\nabla \phi_{\alpha, \beta} = \nabla u_{\alpha, \beta} + \beta \nabla \eta \text{ and } \nabla u = \nabla u_{\alpha, \beta} \text{ when }
\nabla u_{\alpha, \beta} \neq 0.
$$
We then multiply $\phi_{\alpha, \beta}$ to the equation \eqref{Equ:right-yamabe-equation} and get
$$
\frac {4(n-1)}{n-2}\int_{\Sigma_\alpha} \nabla u\cdot \nabla\phi_{\alpha, \beta} dvol
= \int_{\Sigma_\alpha} ( - R u + R[g]u^\frac {n+2}{n-2})\phi_{\alpha, \beta}dvol.
$$
Therefore
\begin{equation}\label{Equ:newton-capacity-estimate}
\aligned
\frac {4(n-1)}{n-2} & \int_{\Sigma_\alpha}  |\nabla u_{\alpha, \beta}|^2 dvol  \\
& = \beta  \int_{\Sigma_\alpha} (\frac {n-2}{4(n-1)} \nabla u \cdot\nabla \eta + (-Ru + R[g]u^\frac {n+2}{n-2})(1-\eta))dvol \\
& - \int_{\Sigma_\alpha} (-Ru + R^+[g]u^\frac {n+2}{n-2})(\beta - u_{\alpha, \beta}) dvol\\
& + \int_{\Sigma_\alpha} R^-[g]u^\frac {n+2}{n-2} (\beta - u_{\alpha, \beta}) dvol\\
& \leq C\beta,
\endaligned
\end{equation}
where $C$ depends on $\alpha$ and $\eta$ but does not depend on $\beta$, due the support of $1 - \eta$ and \eqref{Equ:L^1-R^-}. 
That is
$$
\int_{\Sigma_\alpha} |\nabla \frac {u_{\alpha, \beta}}\beta |^2 dvol \leq \frac C\beta \to 0
$$
as $\beta\to \infty$, where $\frac {u_{\alpha, \beta}}\beta$ is a function that is identically one in a neighborhood of $S$. This implies
the Newton capacity $\text{Cap}_2(S, D)$ of $S$ is zero.
Consequently, we know $S$ is of the Hausdorff dimensions not greater than $n-2$ 
(cf. \cite{AM} and \cite[Theorem 2.10 in Chapter VI]{SYb}). So the proof is complete.
\end{proof}


\subsection{$-\Delta u$ is a Radon measure on $D$}
In order to use the estimates of potentials in the previous section, we need the following lemma (cf. \cite[Lemma 3.2 - 3.4]{MQ-g}).   

\begin{lemma} \label{Lem:radon-measure} 
Let $(M^n, \bar g)$ be a compact Riemannian manifold and $S$ be a 
closed subset in $M^n$. And let $D$ be an open neighborhood of $S$ where the scalar curvature $R[\bar g]$ is nonpositive.  
Suppose that $g = u^\frac 4{n-2} \bar g$ is a conformal metric on $D\setminus S$ and is geodesically complete near $S$. 
Then $-\Delta u$ is a Radon measure on $D$ and $-\Delta u|_S \geq 0$, provided that
$$R^-[g] \in L^\frac {2n}{n+2}(D\setminus S, g) \bigcap L^p(D\setminus S, g)$$ for some $p> n/2$.
\end{lemma} 
  
\begin{proof} Again, recall the scalar curvature equation
\begin{equation}\label{Equ:scalar-curv-equation-f}
-\frac {4(n-1)}{n-2}\Delta u = -R u + R^+[g] u^\frac {n+2}{n-2} - R^-[g] u^\frac {n+2}{n-2} = f \text{ in $D\setminus S$},
\end{equation}
where
$$
\int_{D} R^-[g] u^\frac {n+2}{n-2} dvol < \infty.
$$
And, in the light of Lemma \ref{Lem:lemma-3.1}, we know
$$
u(x) \to \infty \text{ as $x\to S$}.
$$
Then we claim the right hand side $f$ of the equation \eqref{Equ:scalar-curv-equation-f} is in $L^1(D)$.
To prove this claim, we follow the argument in the proof of \cite[Theorem 3.2]{MQ-g} 
(stated as \cite[Lemma 3.2]{MQ-g}) . Let 
$$
\alpha_s (t) = \left\{\aligned t & \quad t \leq s;\\
\text{increasing} & \quad t \in [s, 10 s];\\
2s & \quad t\geq 10 s \endaligned\right.
$$ 
(this functions was used in \cite{DHM}). 
Notice that one may require $\alpha_s'\in [0, 1]$ and $\alpha_s'' \leq 0$. We calculate
$$
- \Delta \alpha_s(u) = - \alpha''(u) |\nabla u|^2 + \alpha_s'(u)(-\Delta u) 
$$
and, for $s> \max \{u(x): x\in \partial D\}$,
$$
\int_{\partial D} \frac {\partial u}{\partial \nu} d\sigma
= \int_D  \Delta \alpha_s(u) dvol = 
\int_D ( -  \alpha''(u) |\nabla u|^2 + \alpha'(u) \frac {n-2}{4(n-1)} f)dvol.
$$
Hence
$$
\int_D ( - \alpha_s''(u)|\nabla u|^2 + \alpha_s'(u) \frac {n-2}{4(n-1)} f^+) \, dvol \leq 
\int_{\partial D} \frac {\partial u}{\partial \nu} d\sigma 
+ \frac {n-2}{4(n-1)}\int_{D} R^-[g] u^\frac {n+2}{n-2} dvol
$$
and
$$
\int_D |\Delta \alpha_s(u)| dvol = \int_D (- \alpha''(u) |\nabla u|^2+ \alpha'(u) \frac {n-2}{4(n-1)}
(f^+ + f^-) )dvol.
$$
By Fatou's lemma, as $s\to\infty$, we have
$$
 \int_{D} f^+dvol \leq \frac {4(n-1)}{n-2}
 \int_{\partial D} \frac {\partial u}{\partial \nu} d\sigma  + \int_{D} R^-[g] u^\frac {n+2}{n-2} dvol.
$$
So the claim is proven. Moreover, 
$$
\int_D |\Delta \alpha_s(u)| dvol
\leq \int_{\partial D} \frac {\partial u}{\partial \nu} d\sigma  + \frac {n-2}{2(n-1)} 
\int_{D} R^-[g] u^\frac {n+2}{n-2} dvol.
$$
Consequently, for $\phi\in C^\infty_c(D)$,
$$
\aligned
|-\Delta & \alpha_s(u) (\phi) |  = |\int_D (-\Delta\alpha_s (u))\phi dvol | \\
& \leq \int_D|\Delta \alpha_s(u)| dvol \|\phi\|_{C^0(D)}\\
& \leq (\int_{\partial D} \frac {\partial u}{\partial \nu} d\sigma
+ \frac {n-2}{2(n-1)} \int_{D\setminus S} R^-[g] u^\frac {n+2}{n-2} dvol) \|\phi\|_{C^0(D)},
\endaligned
$$
for any $s$ larger. Before we show $-\Delta u$ is a Radon measure, let us state and prove a lemma which is useful
for the proof now and later in the following sections.

\begin{lemma}\label{Lem:L^1-gradient}
Let $(M^n, \bar g)$ be a compact Riemannian manifold and $S$ be a 
closed subset in $M^n$. And let $D$ be an open neighborhood of $S$ where the scalar curvature $R[\bar g]$ is nonpositive.  
Suppose that $g = u^\frac 4{n-2} \bar g$ is a conformal metric on $D\setminus S$ and is geodesically complete near $S$. 
Then 
\begin{equation}\label{Equ:L^1-gradient}
\nabla u\in L^p (D) \text{ and } u \in L^q (D)
\end{equation}
for $p\in [1, \frac n{n-1})$ and $q\in [1, \frac n{n-2})$, provided that
$$R^-[g] \in L^\frac {2n}{n+2}(D\setminus S, g) \bigcap L^p(D\setminus S, g)$$ for some $p> n/2$.
\end{lemma} 

\begin{proof}
In fact, we continue from the above, for $\phi\in C^\infty_c(D)$,
\begin{equation}\label{Equ:L^p-gradient}
\aligned
|\int_D & \nabla \alpha_s(u) \cdot \nabla \phi dvol | = |\int_D (-\Delta[\bar g]\alpha_s(u)\phi) dvol|\\
&  \leq  (\int_{\partial D} \frac {\partial u}{\partial \nu} d\sigma
+ \frac {n-2}{2(n-1)} \int_{D\setminus S} R^-[g] u^\frac {n+2}{n-2} dvol) \|\phi\|_{C^0(D)}\\
& \leq C (\int_{\partial D} \frac {\partial u}{\partial \nu} d\sigma 
+ \frac {n-2}{2(n-1)} \int_{D\setminus S} R^-[g] u^\frac {n+2}{n-2} dvol) \|\nabla \phi\|_{L^\lambda(D)}
\endaligned
\end{equation}
for any $\lambda > n$ due to the Sobolev embedding theorem. Therefore, for any $s$ appropriately large,
$$
\|\nabla \alpha_s(u)\|_{L^{p}(D)} \leq C \text{ and } \|\alpha_s(u)\|_{L^q(D)}\leq C
$$
for some constant $C$ and $p = \lambda' \in (1, \frac n{n-1})$ and $q\in [1, \frac n{n-2})$, where $C$ is independent of $s$. 
Therefore we first have, by Fatou's lemma, 
$$
\|u\|_{L^q(D)} \leq C
$$
for some $C$ and $q\in [1, \frac n{n-2})$. Moreover, we calculate
\begin{equation}\label{Equ:integrable-gradient-u}
\aligned
|\nabla u (\phi) | & = |\int_D u \nabla \phi dvol|= |\lim_{s\to\infty} \int_D \alpha_s(u) \nabla\phi dvol| = 
|\lim_{s\to\infty} \int_D \alpha_s'(u) \nabla u\phi dvol| \\
& \leq  \limsup_{s\to \infty}\|\alpha'_s(u) \nabla u\|_{L^{p}(D)}  \|\phi\|_{L^\lambda} \leq C \|\phi\|_{L^\lambda}.
\endaligned
\end{equation}
This implies 
\begin{equation}\label{Equ:gradient-u-integrable}
\nabla u\in L^p (D) \text{ and } u \in L^q (D)
\end{equation}
for $p\in [1, \frac n{n-1})$ and $q\in [1, \frac n{n-2})$. This lemma is proven.
\end{proof} 

Back to the proof of Lemma \ref{Lem:radon-measure}, 
\begin{equation}\label{Equ:distributional-sense}
\aligned
(- \Delta u )(\phi) & = \int_D \nabla u\cdot \nabla \phi \, dvol\\
& = \lim_{s\to \infty} \int_D \alpha_s'(u) \nabla u\cdot\nabla \phi \, dvol\\
& = \lim_{s\to\infty} (-\Delta\alpha_s(u)) (\phi))
\endaligned
\end{equation}
where the dominated convergence theorem is applied due to $\nabla u\in L^1(D)$. Thus, for $\phi\in C^\infty_c(D)$, 
$$
|(- \Delta u) (\phi) | \leq \ (\int_{\partial D} \frac {\partial u}{\partial \nu} d\sigma
+ \frac {n-2}{2(n-1)} \int_{D\setminus S} R^-[g] u^\frac {n+2}{n-2} dvol) \|\phi\|_{C^0(D)},
$$
which implies that $-\Delta u$ is a Radon measure on $D$. To show that $-\Delta u|_S \geq 0$, we calculate,
for a nonnegative function $\phi\in C^\infty_c(D)$,
$$
\aligned
(-\Delta u) (\phi) & = \int_D \nabla u \cdot\nabla\phi dvol = \lim_{s\to\infty} \int_D \nabla\alpha_s(u)\cdot\nabla\phi dvol\\
&  = \lim_{s\to\infty} \int_D (-\Delta\alpha_s(u))\phi dvol\\
& = \lim_{s\to\infty}\int_D [(\alpha_s'(u) (\frac {n-2}{4(n-1)}(-Ru + R[g] u^\frac {n+2}{n-2}) - \alpha_s''(u) |\nabla u|^2] \phi 
dvol \\
& \geq - \{\frac {n-2}{4(n-1)} \int_{\text{supp}\phi\setminus S} |-R u + R[g]u^\frac {n+2}{n-2}| dvol\}\|\phi\|_{C^0(D)}\to 0
\endaligned
$$ 
as $\int_{\text{supp}\phi\setminus S}dvol \to 0$ and $\|\phi\|_{C^0(D)}=1$, which implies $-\Delta u|_S \geq 0$.
\end{proof}


\subsection{Main result on the Hausdorff dimensions}
Now we are ready to state and prove our result on the Hausdorff dimension of the singular set $S$, which is a significant
improvement of Proposition \ref{Prop:newton-capacity}. For convenience of readers, we recall 
Theorem \ref{Thm:intr-main-thm-3} from the introduction. 
  
\begin{theorem} \label{Thm:main-result-R} Let $(M^n, \bar g)$ be a compact Riemannian manifold and $S$ be a 
closed subset in $M^n$. And let $D$ be an open neighborhood of $S$.  
Suppose that $g = u^\frac 4{n-2} \bar g$ is a conformal metric on $D\setminus S$ and is geodesically complete near $S$. 
 Then the Hausdorff dimension
\begin{equation}\label{Equ:hausdorff-dimension} 
\text{dim}_\mathscr{H} (S) \leq \frac {n-2}2
\end{equation}
provided $R^-[g] \in L^\frac {2n}{n+2} (D\setminus S, g) \bigcap L^p(D\setminus S, g)$
for some $p > n/2$. Consequently, \eqref{Equ:hausdorff-dimension} holds when the scalar curvature $R[g]$ of the 
conformal metric $g$ is nonnegative.
\end{theorem}

\begin{proof} The outline of the proof is as follows: We first show that one may assume the scalar curvature $R[\bar g]$ 
is nonpositive without loss of generality for our purpose. Then we use the Green function to construct the integral representation of the 
solution to the Laplace equation. Finally we apply Lemma \ref{Lem:radon-measure}, Theorem \ref{Thm:main-estimate}, 
and the geodesic completeness to complete the proof .

\vskip 0.1in\noindent
{\bf Step I}\quad  In this step, we find a conformal change $\bar h = v^\frac 4{n-2} \bar g$ such that the scalar curvature $R[\bar h]$ is
nonpositive (or even negative) in $D$, based on the similar idea used in the proof of \cite[Lemma 3.1]{MQ-g}. This is trivial if the 
Yamabe constant 
of $(M^n, \bar g)$ is nonpositive. Otherwise, take a point $p\in M^n\setminus D$ and consider a connected sum of $M^n$ with 
another compact Riemannian manifold
$(M^n_1, \bar g_1)$ with very negative Yamabe constant in such way that the conformal structure on the connected sum
$M^n\sharp M^n_1$ is unchanged in $D\subset M^n\sharp M^n_1$. Then, by \cite[Theorem 5]{Gil}, the Yamabe constant of
such connected sum is negative. Therefore one easily finds a conformal metric $\bar h = v^\frac 4{n-2}\bar g$ whose scalar curvature
is negative in $D$, where the function $v\in C^\infty(\bar{D})$ and 
\begin{equation}\label{Equ:lower-upper-v}
C^{-1} \leq v \leq C \text{ in $\bar {D}$}
\end{equation}
for some positive constant $C$. In any cases, we have $g = u^\frac 4{n-2}\bar g = (\frac uv)^\frac 4{n-2} \bar h$ 
and the scalar curvature $R[\bar h]$ is nonpositive. In conclusion, due to \eqref{Equ:lower-upper-v}, 
we may simply assume $R[\bar g]$ is nonpositive (or even negative) in $D$ without
loss of any generality for the purpose of obtaining the growth estimate like the one
given in Theorem \ref{Thm:main-estimate}.

\vskip 0.1in\noindent
{\bf Step II} \quad In this step, we use the Green function to construct the integral representation of the solution $u$. In the light of
Lemma \ref{Lem:radon-measure}, we may write
$$
-\Delta u =  \mu \text{ in $D$}
$$
for a Radon measure $\mu$ on $D$. Let $G(x, y)$ be the Green function on $D$ given by \cite[Theorem 4.17]{Au}. Then
$$
u = \int_D G(x, y)d\mu(y) + h
$$
for a smooth function $h$ that is harmonic in $D$. By \cite[Theorem 4.17 (c)]{Au}, we have
$$
0< G(x, y) \leq  \frac C{d(x, y)^{n-2}}
$$
for some constant $C$ and $x, y\in D$. We therefore arrive at, for $x\in D$, 
\begin{equation}\label{Equ:decomposition-u}
u(x)  \leq \int_D G(x, y)d\mu^+ + h(x) \leq C\mathscr{R}^{2, D}_{\mu^+} (x) + h(x).
\end{equation}

\noindent {\bf Step III}\quad Assume otherwise that $\text{dim}_{\mathscr{H}} (S) = d > \frac {n-2}2$. 
From Theorem \ref{Thm:main-estimate}, there is a point $p\in S$ such that 
$$
\mathscr{R}^{2, D}_{\mu^+} (x) \leq \frac C{d(x, p)^{n-2-d}}
$$
at least for $x$ along a short geodesic ray $\gamma$ from $p$, which implies 
\begin{equation}\label{Equ:punch-line}
u(x)^\frac 2{n-2}  \leq \frac C{d(x, p)^\frac {2(n-2-d)}{n-2}}
\end{equation}
at least for $x$ along a short geodesic ray $\gamma$ from $p$, where
$$
\frac {2(n-2-d)}{n-2} = 2 - \frac {2d}{n-2} < 1
$$
when $d> \frac {n-2}2$. Now the length of the curve $\gamma$ with respect to the conformal metric $g = u^\frac 4{n-2} \bar g$ is
$$
L(\gamma, g) \leq  C \int_0^{l_0} \frac 1{s^\frac {2(n-2-d)}{n-2}} ds < \infty
$$
when $d> \frac {n-2}2$, which contradicts with the geodesic completeness of the conformal metric $g = u^\frac 4{n-2}\bar g$. 
The proof is complete.
\end{proof}

The study of singular solutions to the scalar curvature equations started from the seminal paper \cite{SY} (see also 
\cite[Chapter VI]{SYb} \cite{Ca} and  \cite{Sc, MS, MP}) on domains of the sphere. 
Theorem \ref{Thm:main-result-R} here can be considered as a necessary condition for the existences of singular solutions on
domains in general Riemannian manifolds and compared with \cite[Theorem 2.7]{SY} and \cite[Theorem C]{Ca}, which stated the 
similar result for domains in the round sphere $S^n$ and slightly stronger curvature assumptions. Clearly \cite[Proposition 2.4]{SY} 
and the quantity $d(M)$ there are not of local nature, while our approach here is very much local in nature.


\section{On $Q$-curvature equations}\label{Sec:Q-curvature}

In this section we will use linear potential theory developed in Section \ref{Sec:potential-capacity} 
to study $Q$-curvature equations and prove our results on the Hausdorff dimensions of the singular sets of
positive solutions of $Q$-curvature equations which correspond to ends of complete conformal metrics on domains of a compact
Riemannian manifold. 

Again we remark here that all of the results in this section hold if we assume $S$ is compact, $D\subset M^n$ is a bounded domain that
contains $S$, and $(M^n, \bar g)$ is just complete. Because the possible noncompact part $M^n\setminus \bar D$ is not relevant 
for the purpose here.


\subsection{$Q$-curvature equations in dimensions greater than $4$} In this subsection we focus on the equation 
\eqref{Equ:Q-curvature-equation} in dimensions greater than $4$. We will always assume that the scalar curvature of  
the conformal metric $g= u^\frac 4{n-4} \bar g$ is nonnegative. We will first prove some preliminary estimates based on 
discussions in the previous section. Our strategy is to consider the bi-Laplace operator as the composition of the Laplace 
operators. Let us write the scalar curvature equation and its consequence:
\begin{equation}\label{Equ:scalar-curvature-v}
-\Delta u^\frac {n-2}{n-4} + \frac {n-2}{4(n-1)} Ru^\frac {n-2}{n-4} = \frac {n-2}{4(n-1)}R[g]u^\frac {n+2}{n-4} \text{ in $D\setminus S$}
\end{equation}
and
\begin{equation}\label{Equ:laplace-u}
- \Delta u =  \frac {2}{n-4} \frac {|\nabla u|^2}u +\frac {n-4}{4(n-1)}(- Ru +  R[g] u^\frac {n}{n-4})  \quad \text{ in $D\setminus S$}.
\end{equation}
Here, and from now on in the following, all geometric quantities are under the background metric $\bar g$ unless indicated otherwise.

\begin{lemma}\label{Lem:prelim-estimate}
Let $(M^n, \bar g)$ be a compact Riemannian manifold for $n\geq 5$ and $S$ be a 
closed subset in $M^n$. And let $D$ be an open neighborhood of $S$ where the scalar curvature $R \leq -c_0 < 0$.  
Suppose that $g = u^\frac 4{n-4} \bar g$ is a conformal metric on $D\setminus S$ with nonnegative scalar curvature $R[g]\geq 0$ 
and is geodesically complete near $S$. And suppose also that
$$
Q_4^-[g] \in L^\frac {2n}{n+4}(D\setminus S, g).
$$
Then
\begin{equation}\label{Equ:prelim-estimate}
\aligned
& \text{as a function on $D\setminus S$, } -\Delta  u  \to +\infty \text{ as $x\to S$};\\
& \text{as a Radon measure on $D$, }  \Delta u |_S = 0 ;\\
& \text{in fact, } \Delta  u \in L^p(D) \text{ for any $p\in [1, \frac n{n-2})$}.
\endaligned
\end{equation}
\end{lemma}

\begin{proof} First, using Lemma \ref{Lem:L^1-gradient} for $u^\frac {n-2}{n-4}$, we know that 
\begin{equation}\label{Equ:L^p-u}
u\in L^p(D) \text{ for $p\in [1, \frac n{n-4})$}.
\end{equation}
Also, from Lemma \ref{Lem:lemma-3.1} for $u^\frac {n-2}{n-4}$, 
\begin{equation}\label{Equ:infinite}
u(x)\to +\infty \text{ as $x\to S$},
\end{equation}
which implies, by \eqref{Equ:laplace-u},
$$
-\Delta u \to +\infty \text{ as $x\to S$}.
$$
To prove $-\Delta u$ is an integrable function in distributional sense, we first realize $-\Delta u$ is a Radon measure
on $D$ following \eqref{Equ:laplace-u} and Lemma \ref{Lem:radon-measure}.  And, as a side product, we also have
$$
\int_{D} [\frac 2{n-4} \frac {|\nabla u|^2}u + \frac {n-4}{4(n-1)}(- Ru +  R[g] u^\frac {n}{n-4})] dvol< \infty.
$$
In fact, from \eqref{Equ:scalar-curvature-v} and 
Lemma \ref{Lem:radon-measure}, we also know $-\Delta u^\frac {n-2}{n-4}$ is a Radon measure on $D$. 
To use this fact we calculate
$$
-\Delta \alpha_s(u) = -\Delta (\alpha_s(u)^\frac {n-2}{n-4})^\frac {n-4}{n-2} 
= \frac {n-4}{n-2} \alpha_s(u)^{-\frac 2{n-4}}(-\Delta \alpha_s(u)^\frac {n-2}{n-4}) + \frac 2{n-4} \frac {|\nabla\alpha_s(u)|^2}{\alpha_s(u)}.
$$
To prove $-\Delta u|_S = 0$, we consider
$$
(-\Delta u )(\phi) = \int_D \nabla u \cdot\nabla\phi dvol
$$
where $\nabla u$ is integrable in distributional sense directly from \eqref{Equ:integrable-gradient-u} 
and \eqref{Equ:gradient-u-integrable}. Therefore
$$
\int_D \nabla u \cdot\nabla\phi dvol = \lim_{s\to\infty} \int_D \alpha_s'(u) \nabla u \cdot\nabla\phi dvol = 
\lim_{s\to\infty}\int_D (-\Delta \alpha_s(u) )\phi dvol
$$
and
$$
\aligned
(-\Delta u)(\phi) & =  \frac {n-4}{n-2} \lim_{s\to\infty} \int_D \alpha_s(u)^{-\frac 2{n-4}}(-\Delta \alpha_s(u)^\frac {n-2}{n-4})\phi dvol\\
& \quad  + \frac 2{n-4} \lim_{s\to\infty}\int_D \frac {|\nabla\alpha_s(u)|^2}{\alpha_s(u)}\phi dvol\\
& =  \frac {n-4}{n-2} u^{-\frac 2{n-4}}(-\Delta u^\frac {n-2}{n-4})(\phi)  + \frac 2{n-4} \int_D \frac {|\nabla u|^2}{u}\phi dvol \\
& \to 0
\endaligned
$$
as $\int_{\text{supp}\phi\setminus S}dvol[\bar g] \to 0$ and $\|\phi\|_{C^0(D)}\leq 1$. 
The proof will be complete after the following $L^p$ estimate. To get the $L^p$ estimate, we first calculate
\begin{equation}\label{Equ:L^1-Q^-}
\aligned
\int_{D\setminus S} Q_4^-[g] u^\frac {n+4}{n-4} dvol & 
= (\int_{D\setminus S} (Q_4^-[g])^\frac {2n}{n+4} u^\frac {2n}{n-4}dvol)^\frac {n+4}{2n}
\text{vol}(D)^\frac {2n}{n-4}\\
& =  (\int_{D\setminus S} (Q_4^-[g])^\frac {2n}{n+4} dvol[g])^\frac {n+4}{2n}
\text{vol}(D)^\frac {2n}{n-4} < \infty.
\endaligned
\end{equation}
Then we continue to use notations in the proof of Proposition \ref{Prop:newton-capacity} and let 
$$
\alpha = \max\{u(x): x\in \partial D\}
$$
and $\alpha < \beta$. And recall
$$
u_{\alpha, \beta} = \left\{\aligned \beta & \quad x\in \Sigma_{\alpha+\beta};\\
u(x) - \alpha & \quad x\in \Sigma_\alpha\setminus \Sigma_{\alpha+ \beta}\endaligned\right.
$$
and 
$$
\phi_{\alpha, \beta} = \left\{\aligned  u_{\alpha, \beta} - \beta \eta & = u - (\alpha+\beta) + \beta(1-\eta), 
 \text{ in $\Sigma_\alpha\setminus\Sigma_{\alpha+\beta}$}\\
& = 0  \text{ on $\partial\Sigma_\alpha$}\\
& = 0  \text{ on $\partial\Sigma_{\alpha+\beta}$}
\endaligned\right.,
$$
where $\eta$ is a fixed cut-off function in $C^\infty_c(\Sigma_\alpha)$ 
and is identically one in a neighborhood of $S$, and $\beta$ is arbitrarily large.
We now first multiply $1- \eta$ to the $Q$-curvature equation 
\eqref{Equ:Q-curvature-equation}, 
integrate over $D$, apply integral by parts multiple times, and get
\begin{equation}\label{Equ:positive-part-integrable}
\int_D (1-\eta)Q_4^+u^\frac {n+4}{n-4} dvol \leq \int_D Q_4^-u^\frac {n+4}{n-4}dvol + C
\end{equation}
for some constant $C$ depending on the cut-off function $\eta$, $u$ at $\partial D$, and $\|u\|_{L^1(D)}$.
We then multiply $\phi_{\alpha, \beta}$ to both sides of the  
$Q$-curvature equation \eqref{Equ:Q-curvature-equation}, integrate over $\Sigma_\alpha\setminus\Sigma_{\alpha+\beta}$, and get
\begin{equation}\label{Equ:integral-estimate}
\aligned
\int_{\Sigma_\alpha\setminus \Sigma_{\alpha+\beta}} &  \Delta u \Delta \phi_{\alpha, \beta} dvol  
- \int_{\partial \Sigma_\alpha} \Delta u \frac {\partial u} {\partial \nu} d\sigma - 
\int_{\partial \Sigma_{\alpha+\beta}} \frac{\partial u}{\partial\nu} \Delta u d\sigma\\
& -  \int_{\Sigma_\alpha\setminus\Sigma_{\alpha+\beta}} (4 A(\nabla u, \nabla \phi_{\alpha, \beta})  
- (n-2)J\nabla u\cdot\nabla\phi_{\alpha, \beta}) dvol \\
& + \frac {n-4}2 \int_{\Sigma_\alpha\setminus\Sigma_{\alpha+\beta}} Q u\phi_{\alpha, \beta} dvol
=  \frac {n-4}2 \int_{\Sigma_\alpha\setminus\Sigma_{\alpha+\beta}}  Q_4[g] u^\frac {n+4}{n-4}\phi_{\alpha, \beta} dvol,
\endaligned
\end{equation}
where $\nu$ is the outward normal direction at the boundary and the boundary term 
$\int_{\partial \Sigma_{\alpha+\beta}} \frac{\partial u}{\partial\nu} (-\Delta u) d\sigma$ is nonnegative due
to \eqref{Equ:laplace-u} and $\frac {\partial u}{\partial \nu}|_{\partial \Sigma_{\alpha+\beta}} = |\nabla u|$.
Therefore, 
\begin{equation}\label{Equ:integral-estimate-2}
\aligned
\int_{\Sigma_\alpha\setminus \Sigma_{\alpha+\beta}}  & (\Delta u)^2 dvol + \beta 
\int_{\Sigma_\alpha\setminus \Sigma_{\alpha+\beta}} (\Delta u) (\Delta (1-\eta)) dvol\\
& \leq - \int_{\partial \Sigma_\alpha} (-\Delta u) \frac {\partial u} {\partial \nu} d\sigma  +  
C\int_{\Sigma_\alpha\setminus \Sigma_{\alpha+\beta}} |\nabla u|^2 dvol  \\
& \quad - \beta \int_{\Sigma_\alpha} (4 A(\nabla u, \nabla (1-\eta))  
- (n-2)J\nabla u\cdot\nabla (1-\eta)) dvol \\
& \quad + C \beta \int_D u\, dvol  +  C\beta \int_D  Q_4^-[g] u^\frac {n+4}{n-4} dvol,
\endaligned
\end{equation}
where we use \eqref{Equ:positive-part-integrable} and 
$|\phi|\leq \beta$ in $\Sigma_\alpha\setminus\Sigma_{\alpha+\beta}$. 
After applying integral by parts, we get, 
\begin{equation}\label{Equ:lead-to-dim}
\int_{\Sigma_\alpha\setminus \Sigma_{\alpha+\beta}} |\Delta u|^2 dvol \leq  C 
 \int_{\Sigma_\alpha\setminus \Sigma_{\alpha+\beta}}|\nabla u|^2 dvol + C\beta
\end{equation}
for some constant $C$ depending on the cut-off function $\eta$, $u$ at $\partial \Sigma_\alpha$, and $\|u\|_{L^1(D)}$, because 
$$
\int_{\Sigma_\alpha\setminus\Sigma_{\alpha+\beta}} \Delta u \Delta \eta dvol
= \int_{\Sigma_\alpha\setminus\Sigma_{\alpha+\beta}}  u \Delta^2\eta dvol
$$
and similarly we may unload all derivatives from $u$ by integral by parts for the other terms in the above 
\eqref{Equ:integral-estimate-2}. Now, to get a priori estimate, we calculate
$$
\aligned
\int_{\Sigma_\alpha\setminus\Sigma_{\alpha+\beta}} |\nabla u|^2 dvol & \leq \frac 1{(n-4)C} 
\int_{\Sigma_\alpha\setminus \Sigma_{\alpha+\beta}} 
\frac {|\nabla u|^4}{u^2} dvol + \frac {(n-4)C}{4} \int_{\Sigma_\alpha\setminus \Sigma_{\alpha+\beta}}  u^2 dvol\\
& \leq \frac 1{2C} \int_{\Sigma_\alpha\setminus \Sigma_{\alpha+\beta}} |\Delta u|^2 dvol
+ \frac {(n-4)C}{4} (\alpha+\beta) \int_D  u\, dvol,
\endaligned
$$
due to \eqref{Equ:laplace-u}, which implies, from \eqref{Equ:lead-to-dim},
\begin{equation}\label{Equ:important-step}
\int_{\Sigma_\alpha\setminus \Sigma_{\alpha+ \beta}} |\Delta u|^2 dvol \leq C(\alpha+\beta).
\end{equation}
We claim that \eqref{Equ:important-step} implies 
\begin{equation}\label{Equ:iteration}
\Delta u \in L^p(D)
\end{equation}
for all $p\in [1, \frac n{n-2})$. To prove \eqref{Equ:iteration}, we first derive from \eqref{Equ:important-step}, 
$$
2^{-i} \int_{\Sigma_{2^{i-1}}\setminus\Sigma_{2^{i}}} |\Delta u|^2 dvol\leq C,
$$
for $i\geq i_0$ large, which implies
$$
\int_{\Sigma_{2^{i-1}}\setminus\Sigma_{2^{i}}} \frac {|\Delta u|^2}u dvol \leq 2C
$$
and, for $s>0$ appropriately small for any $p\in [1, \frac n{n-2})$, 
$$
\aligned
\int_{\Sigma_{2^{i_0-1}} \setminus S} & \frac {|\Delta u|^2}{u^{1+s}} dvol  = 
\sum_{i=i_0}^\infty \int_{\Sigma_{2^{i-1}}\setminus\Sigma_{2^i}} \frac {|\Delta u|^2}{u^{1+ s}} dvol \\
& \leq \sum_{i=i_0}^\infty 2^{s(-i+1)} \int_{\Sigma_{2^{i-1}}\setminus\Sigma_{2^{i}}} \frac {|\Delta u|^2}u dvol < \infty.
\endaligned
$$
Thus
$$
\int_{D\setminus S} |\Delta u|^p dvol \leq (\int_{D\setminus S} \frac {|\Delta u|^2}{u^{1+s}} dvol)^\frac p2 (\int_{D\setminus S} 
u^\frac {(1+s)p}{2-p} dvol)^{1 - \frac p2}< \infty
$$
where
$$
\frac {(1+s)p}{2-p} < \frac n{n-4}.
$$
\end{proof}

\begin{corollary}\label{Cor:prelim-dim} Under the same assumptions as in Lemma \ref{Lem:prelim-estimate} we have
$$
\text{dim}_{\mathscr{H}} (S) \leq n-4.
$$
\end{corollary}

\begin{proof} Consequently from \eqref{Equ:laplace-u} and \eqref{Equ:important-step}, we have
$$
\aligned
\int_{\Sigma_\alpha\setminus \Sigma_{\alpha+\beta}} |\nabla \frac {u_{\alpha, \beta}}\beta |^4dvol 
& \leq \frac {(\alpha+\beta)^2}{\beta^4} \int_{\Sigma_\alpha\setminus \Sigma_{\alpha+\beta}}  \frac {|\nabla u|^4}{u^2} dvol\\
& \leq \frac {(\alpha+\beta)^2}{\beta^4}  \int_{\Sigma_\alpha\setminus \Sigma_{\alpha+\beta}} |\Delta u|^2 dvol \\
& \leq C \frac {(\alpha+\beta)^3}{\beta^4} 
\endaligned
$$
for some $\alpha$ appropriately large and $\beta\to\infty$, 
which leads to $\text{Cap}_4(S) = 0$ and completes the proof similar to the proof of Proposition \ref{Prop:newton-capacity}
 (cf. \cite{AM} and \cite[Theorem 2.10 in Chapter VI]{SYb}).
\end{proof} 

\begin{lemma}\label{Lem:radon-measure-Q}
Let $(M^n, \bar g)$ be a compact Riemannian manifold for $n\geq 5$ and $S$ be a 
closed subset in $M^n$. And let $D$ be an open neighborhood of $S$ where the scalar curvature $R[\bar g] \leq -c_0 < 0$.  
Suppose that $g = u^\frac 4{n-4} \bar g$ is a conformal metric on $D\setminus S$ with nonnegative scalar curvature $R[g]\geq 0$ 
and is geodesically complete near $S$. And suppose also that
$$
Q_4^-[g] \in L^\frac {2n}{n+4}(D\setminus S, g).
$$
Then $\Delta^2  u$ is a Radon measure on $D$ and $\Delta^2  u|_S \geq 0$.
\end{lemma}

\begin{proof} Let $v = -\Delta u$. We will follow the proof of Lemma \ref{Lem:radon-measure} to show that $-\Delta v$ is a Radon
measure on $D$ using Lemma \ref{Lem:prelim-estimate}. 
We continue to use the notations from the proof of Lemma \ref{Lem:radon-measure}. We calculate
$$
- \Delta \alpha_s(v) =  \alpha'(v) (-\Delta v) - \alpha''(v) |\nabla v|^2,
$$ 
where, by the $Q$-curvature equation \eqref{Equ:Q-curvature-equation}, we have
$$
-\Delta v = - \text{div} (4A(\nabla u) -  (n-2)J\nabla u) - \frac {n-4}2 Q_4u + Q_4[g] u^\frac {n+4}{n-4} \text{ in $D\setminus S$}
$$
and
$$
-\Delta \alpha_s(v) = -\alpha_s''(v)|\nabla v|^2 + \alpha'(v)(
- \text{div} (4A(\nabla u) -  (n-2)J\nabla u) - \frac {n-4}2 Q_4u + Q_4[g] u^\frac {n+4}{n-4}) 
$$
in $D$. In the light of Lemma \ref{Lem:prelim-estimate}, terms in the right hand side of the above equation are all integrable 
except $-\alpha''(v)|\nabla v|^2 + Q_4^+[g]u^\frac {n+4}{n-4}$. Therefore the argument in the proof of 
Lemma \ref{Lem:radon-measure} works from this point and completes the proof.  
\end{proof}

We now are ready to state and prove our main results for $Q$-curvature equations in dimensions greater than $4$. For this,
we recall Theorem \ref{Thm:intr-main-thm-4} from the introduction.

\begin{theorem}\label{Thm:main-result-Q}
Let $(M^n, \bar g)$ be a compact Riemannian manifold for $n\geq 5$ and $S$ be a 
closed subset in $M^n$. And let $D$ be an open neighborhood of $S$. Suppose that $g = u^\frac 4{n-4} \bar g$ is a 
conformal metric on $D\setminus S$ with nonnegative scalar curvature $R[g]\geq 0$ and is geodesically complete 
near $S$. And suppose also that
$$
Q_4^-[g] \in L^\frac {2n}{n+4}(D\setminus S, g).
$$
Then 
$$
\text{dim}_{\mathscr{H}} (S) \leq \frac {n-4}2.
$$
\end{theorem}

\begin{proof} In the light of Step I in the proof of Theorem \ref{Thm:main-result-R}, we may assume the scalar curvature
$R \leq -c_0 <0$ for some $c_0$ without loss of any generality. Then we use Lemma \ref{Lem:prelim-estimate} and 
\ref{Lem:radon-measure-Q} and conclude that
$$
\Delta^2 u = \mu
$$
for a Radon measure $\mu$ on $D$. We use \cite[Theorem 4.7]{Au} first to write
$$
-\Delta u = \int_D G(x, y) d\mu + h(x) 
$$
for some harmonic function $h(x)$, where $G(x, y)$ is the Green function for $-\Delta$. Then we have
$$
u(x) = \int_D G(x, z) \int_D G(z, y)d\mu(y) dvol(z) + b(x)
$$
where $b(x)$ is bi-harmonic, where
$$
 \int_D G(x, z) \int_D G(z, y)d\mu(y) dvol(z)= \int_D (\int_D G(x, z)G(z, y)dvol(z))d\mu(y)
$$
and
$$
0 < \int_D G(x, z)G(z, y)dvol(z) \leq \frac C{d(x, y)^{n-4}} 
$$
for constant $C$ and $n\geq 5$ due to \cite[Proposition 4.12]{Au}, where \cite[Proposition 4.12]{Au} can be easily proven to be
available for bounded domains in Riemannian manifolds. Hence
$$
u(x) \leq  \mathscr{R}^{4, D}_{\mu^+} (x) + b(x).
$$
From now on, using the same argument of the proof of Theorem \ref{Thm:main-result-R}, 
based on Theorem \ref{Thm:main-estimate} for $\alpha = 4$ and $n\geq 5$, we conclude 
$$
\text{dim}{_\mathscr{H}} (S)\leq \frac {n-4}2
$$
and finish the proof.
\end{proof}

There have been a lot of  works on the study of singular solutions to $Q$-curvature equations on manifolds of dimensions greater
than $4$,  notably \cite{QR1, QR2, CHY, GMS}, for example. Theorem \ref{Thm:main-result-Q}, for instance, is an improvement of \cite[Theorem 1.2]{CHY} in terms of curvature conditions. And the approach here is different from \cite{CHY}.


\subsection{$Q$-curvature equations in dimension $4$} In this subsection we will study the $Q$-curvature
equation \eqref{Equ:Q-curvature-equation-4}. Our approach here in principle is similar to that in the previous subsection but
different in calculations and details. We will always 
assume that the scalar curvature of the conformal metric $g = e^{2u}\bar g$ is nonnegative. We will first derive
some preliminary estimates from the scalar curvature equation for $w=e^u$ and the 
$Q$-curvature equation \eqref{Equ:Q-curvature-equation-4} for $u$. Let us write the scalar curvature equation for $e^u$
\begin{equation}
-\Delta e^u = \frac 16 (-Re^u + R[g]e^{3u}) \text{ in $D\setminus S$}
\end{equation}
and consequently,
\begin{equation}\label{Equ:laplace-u-4}
-\Delta u = |\nabla u|^2 + \frac 16 (-R + R[g] e^{2u})  \text{ in $D\setminus S$}.
\end{equation}

\begin{lemma}\label{Lem:prelim-estimate-4}
Let $(M^4, \bar g)$ be a compact Riemannian manifold and $S$ be a 
closed subset in $M^n$. And let $D$ be an open neighborhood of $S$ where the scalar curvature $R \leq 0$.  
Suppose that $g = e^{2u} \bar g$ is a conformal metric on $D\setminus S$ with nonnegative scalar curvature $R[g]\geq 0$ 
and is geodesically complete near $S$. And suppose also that
$$
Q_4^-[g] \in L^1(D\setminus S, g).
$$
Then
\begin{equation}\label{Equ:prelim-estimate}
\aligned
& \text{as a Radon measure, } -\Delta u |_S = 0;\\
& \text{in fact, } \Delta  u \in L^p(D) \text{ for any $p\in [1, \frac 43)$}.
\endaligned
\end{equation}
\end{lemma}

\begin{proof} First, by Lemma \ref{Lem:lemma-3.1} for $e^u$, we have
$$
u(x) \to \infty \text{ as $x\to S$}.
$$
Then, by the proof of Lemma \ref{Lem:radon-measure} and \eqref{Equ:laplace-u-4}, we know that
\begin{itemize}
\item $-\Delta u$ is a Radon measure on $D$;
\item $\nabla u\in L^p(D)$ for any $p\in [1, \frac 43)$ and $u\in L^p(D)$ for any $p\in [1, 2)$; 
\item $|\nabla u|^2 + \frac 16 (- R + R[g] e^{2u})\in L^1(D)$.
\end{itemize}
The same argument as in the proof of Lemma \ref{Lem:prelim-estimate} we can prove that $-\Delta u|_S = 0$ as a Radon measure.
Also, for the $L^p$ estimate, following the proof of Lemma \ref{Lem:prelim-estimate}, we multiple $1-\eta$ to both sides of 
\eqref{Equ:Q-curvature-equation-4} and get
$$
\int_D (1-\eta)Q_4^+e^{4u} dvol \leq \int_D Q_4^-[g] dvol[g] + C.
$$
Next we multiple $\phi_{\alpha, \beta}$ to both sides of 
\eqref{Equ:Q-curvature-equation-4} and integrate
$$
\aligned
\int_{\Sigma_\alpha\setminus\Sigma_{\alpha+\beta}} & \Delta u\Delta \phi_{\alpha, \beta} dvol   - \int_{\partial \Sigma_\alpha} \Delta u \frac{\partial \phi_{\alpha, \beta}}
{\partial \nu} d\sigma  - \int_{\partial \Sigma_{\alpha+\beta}} \Delta u \frac{\partial \phi_{\alpha, \beta}}
{\partial \nu} d\sigma  \\
&  - \int_{\Sigma_\alpha\setminus\Sigma_{\alpha+\beta}}
(4A(\nabla u, \nabla\phi_{\alpha}) - 2 J\nabla u\cdot\nabla\phi_{\alpha, \beta} dvol\\
& + \int_{\Sigma_\alpha\setminus\Sigma_{\alpha+\beta}}  Q_4 \phi_{\alpha, \beta}dvol 
= \int_{\Sigma_\alpha\setminus\Sigma_{\alpha+\beta}} Q_4[g]e^{4u} \phi_{\alpha, \beta}dvol
\endaligned
$$
and, again, the boundary term at $\partial \Sigma_{\alpha+\beta}$ is with the sign in our favor, thanks 
to \eqref{Equ:laplace-u-4} and $\frac {\partial u}{\partial \nu}|_{\Sigma_{\alpha+\beta}} = |\nabla u|$ for the outward 
normal $\nu$ of $\Sigma_\alpha\setminus\Sigma_{\alpha+\beta}$. Similar to the estimates in the proof of 
Lemma \ref{Lem:prelim-estimate}, we get
$$
\int_{\Sigma_\alpha\setminus\Sigma_{\alpha+\beta}} |\Delta u|^2 dvol \leq 
C \int_{\Sigma_\alpha\setminus\Sigma_{\alpha+\beta}} |\nabla u|^2 dvol + C\beta.
$$
And we handle $\int_{\Sigma_\alpha\setminus\Sigma_{\alpha+\beta}} |\nabla u|^2 dvol$ similarly as before
$$
\int_{\Sigma_\alpha\setminus\Sigma_{\alpha+\beta}} |\nabla u|^2 dvol
\leq \frac 1{2C} \int_{\Sigma_\alpha\setminus\Sigma_{\alpha+\beta}} |\nabla u|^4 dvol
+ C \leq \frac 1{2C} \int_{\Sigma_\alpha\setminus\Sigma_{\alpha+\beta}} |\Delta u|^2 dvol +C.
$$
due to \eqref{Equ:laplace-u-4}. Therefore 
\begin{equation}\label{Equ:lead-to-dim-4}
\int_{\Sigma_\alpha\setminus\Sigma_{\alpha+\beta}} |\Delta u|^2 dvol \leq C\beta.
\end{equation}
Now, using the same idea as in the proof of Lemma \ref{Lem:prelim-estimate},  we rewrite \eqref{Equ:lead-to-dim-4} as
$$
\int_{\Sigma_{2^{i-1}}\setminus\Sigma_{2^{i}}} \frac {|\Delta u|^2}u dvol \leq C
$$
and, for $s>0$ appropriately small for any $p\in [1, \frac 43)$, we derive
$$
\int_{D\setminus S} \frac {|\Delta u|^2}{u^{1+s}} dvol \leq C,
$$
which implies
$$
\int_{D\setminus S} |\Delta u|^p dvol \leq (\int_{D\setminus S} \frac {|\Delta u|^2}{u^{1+s}} dvol)^\frac p2
(\int_{D\setminus S}u^{\frac {(1+s)p}{2-p}} dvol)^{1 - \frac p2}
$$
when
$$
\frac {(1+s)p}{2-p} < 2.
$$
\end{proof}

\begin{corollary}\label{Cor:prelim-dim-4} Under the assumptions as in Lemma \ref{Lem:prelim-estimate-4}, we know
the singular set $S$ is of zero Hausdorff dimension.
\end{corollary}

\begin{proof} From \eqref{Equ:laplace-u-4} and \eqref{Equ:lead-to-dim-4} in the above we have
$$
\int_{\Sigma_\alpha\setminus\Sigma_{\alpha+\beta}} |\nabla \frac {u_{\alpha, \beta}}\beta|^4 dvol \leq C\beta^{-3}
$$
for some $\alpha$ appropriately large and $\beta\to\infty$, which leads to $\text{Cap}_4(S, D) = 0$ and 
completes the proof as in  Proposition \ref{Prop:newton-capacity} (cf. \cite{AM} and \cite[Theorem 2.10 in Chapter VI]{SYb}).
\end{proof}

What follows is to go beyond that $S$ is of zero Hausdorff dimension.  
We now are ready to state and prove our main result on the finiteness of singularities for the $Q$-curvature equation in 
dimension $4$. This is inspired by \cite{Cv, Hu, AH, MQ-a, MQ-g}. We recall Theorem \ref{Thm:intr-main-thm-5} from the 
introduction.

\begin{theorem}\label{Thm:main-result-Q-4}
Let $(M^4, \bar g)$ be a compact Riemannian manifold and $S$ be a 
closed subset in $M^n$. And let $D$ be an open neighborhood of $S$. Suppose that $g= e^{2u}\bar g$ is a conformal metric on $D\setminus S$ with nonnegative scalar curvature $R[g]\geq 0$ and is geodesically complete near $S$. And suppose that
$$
\int_D Q_4^-[g]dvol[g] < \infty.
$$
Then $S$ consists of only finitely many points.
\end{theorem}

\begin{proof} As before, we use the argument in Step I on Theorem \ref{Thm:main-result-R} to assume that the scalar curvature
of the background metric $\bar g$ is less than a negative number, i.e. $R\leq -c_0<0$, without loss of any generality for our
purpose. Let 
$$
v = -\Delta u + u
$$
and claim $-\Delta v$ is a Radon measure on $D$ with $-\Delta v|_S \geq 0$. Let us start with
\begin{equation}\label{Equ:laplace-v}
-\Delta v = \Delta^2 u - \Delta u = - \text{div}(4A(\nabla u) -2J\nabla u) - Q_4 + Q_4[g]e^{4u} - \Delta u.
\end{equation}
By Lemma \ref{Lem:prelim-estimate-4} and \eqref{Equ:laplace-u-4}, we know 
\begin{itemize}
\item $v(x) \to\infty$  as  $x\to S$;
\item All terms in the right side of \eqref{Equ:laplace-v} except $Q_4^+[g]e^{4u}$ is integrable.
\end{itemize}
Therefore, folowing the same argument as in the proof of Lemma \ref{Lem:radon-measure-Q}, the claim is proven. 
Obviously, the same conclusion holds for $\Delta^2 u = -\Delta v + \Delta u$ from $-\Delta v$ and what we know about 
$\Delta u$ in Lemma \ref{Lem:prelim-estimate-4}. Thus we let
$$
\Delta^2 u = \mu
$$
for a Radon measure on $D$ with $\Delta^2 u|_S \geq 0$.  Like in the proof of Theorem \ref{Thm:main-result-Q}, we first write
$$
-\Delta u (x) = \int_D G(x, y)d\mu(y) + h(x)
$$
by \cite[Theorem 4.17]{Au}, where $h(x)$ is a harmonic function. Then we write
$$
u(x) = \int_D G(x, z)\int_D G(z, y)d\mu(y) dvol(z) + b(x) 
$$
where $b(x)$ is a bi-harmonic function and,  due to \cite[Proposition 4.12]{Au}, 
$$
\int_D G(x, z)G(z, y) dvol(z) \leq C(1 + \log \frac 1{d(x, y)})
$$
for some constant $C$ in dimension $4$, where \cite[Proposition 4.12]{Au} can be easily made available on bounded domains in 
manifolds. Therefore 
$$
u(x) \leq C\mathscr{R}^{4, D}_{\mu^+} (x) + b(x).
$$
Applying Theorem \ref{Thm:n-potential}, we have 
$$
\lim_{x\to p \text{ and } x\notin E} \frac {u(x)}{\log\frac 1{d(x, p)}} \leq C\mu^+(\{p\}) = C\mu(\{p\})
$$
where $E$ is a subset that is $n$-thin at $p$. Next, in the light of Theorem \ref{Thm:segment-euclidean}, we conclude
that $\mu(\{p\}) \geq 1/C$ for each $p\in S$ by the completeness of the metric $g$ near $S$, 
which indeed implies that $S$ can only have finitely many points. So the proof is complete.
\end{proof}

Theorem \ref{Thm:main-result-Q-4} is a significant improvement of \cite[Theorem 2]{CQY} (please see also \cite{CH, CQY-d, MQ-g}). 

\vskip 0.3cm
\noindent$\mbox{}^\dag$ Department of Mathematics, Nankai University, Tianjin, China; \\e-mail: 
msgdyx8741@nankai.edu.cn 
\vspace{0.2cm}

\noindent $\mbox{}^\ddag$ Department of Mathematics, University of California, Santa Cruz, CA 95064; \\
e-mail: qing@ucsc.edu

\end{document}